\documentclass[final,hidelinks,onefignum,onetabnum]{siamart220329}

\usepackage{enumitem}
\usepackage{subfig}
\usepackage{amsmath}
\usepackage{bm}
\usepackage{tikz}
\usepackage{appendix}
\usepackage{amsfonts}
\usepackage{multirow}
\usepackage{algorithm,algorithmicx,algpseudocode}

\algdef{SE}[DOWHILE]{Do}{doWhile}{\algorithmicdo}[1]{\algorithmicwhile\ #1}%

\usepackage[left=2.39cm,right=5.81cm,top=3.13cm,bottom=3.20cm,asymmetric]{geometry}

\usepackage{booktabs}

\newsiamremark{remark}{Remark}
\newsiamthm{Def}{Definition}
\newsiamthm{Prop}{Proposition}
\newsiamremark{exmp}{Example}

\ifpdf
\hypersetup{
	pdftitle={},
	pdfauthor={}
}
\fi

\begin{document}

\headers{}{}

\title{Provably Convergent and Robust Newton--Raphson Method: A New Dawn in Primitive Variable Recovery for Relativistic MHD
	\thanks{
			C.~Cai and J.~Qiu were partially supported by National Key R\&D Program of China (No.~2022YFA1004500). K.~Wu was partially supported by Shenzhen Science and Technology Program (No.~RCJC20221008092757098) and NSFC grant (No.~12171227).}}
	
\author{Chaoyi Cai\thanks{
		School of Mathematical Sciences, Xiamen University, Xiamen, Fujian 361005, China (\email{caichaoyi@stu.xmu.edu.cn}).}
	\and Jianxian Qiu\thanks{School of Mathematical Sciences and Fujian Provincial Key Laboratory of Mathematical Modeling and High-Performance Scientific Computing, Xiamen University, Xiamen, Fujian 361005, China (\email{jxqiu@xmu.edu.cn}).}
	\and Kailiang Wu\thanks{Corresponding author. Department of Mathematics and SUSTech International Center for Mathematics, Southern University of Science and Technology, and National Center for Applied Mathematics Shenzhen (NCAMS), Shenzhen, Guangdong 518055, China (\email{wukl@sustech.edu.cn}).}} 

\maketitle
\begin{abstract}
A long-standing and formidable challenge faced by all conservative numerical schemes for relativistic magnetohydrodynamics (RMHD) equations is the recovery of primitive variables from conservative ones. This process involves solving highly nonlinear equations subject to physical constraints. An ideal solver should be ``robust, accurate, and fast—it is at {\em the heart of all} conservative RMHD schemes,'' as emphasized in [S.C.~Noble {\em et al.}, Astrophys.~J., 641 (2006), pp.~626--637]. Despite over three decades of research, seeking efficient solvers that can provably guarantee stability and convergence remains an open problem. 



This paper presents the first theoretical analysis for designing a robust, physical-constraint-preserving (PCP), and provably (quadratically) convergent Newton--Raphson (NR) method for primitive variable recovery in RMHD. Our key innovation is a unified approach for the initial guess, carefully devised based on sophisticated analysis. It ensures that the resulting NR iteration consistently converges and adheres to physical constraints throughout all NR iterations. Given the extreme nonlinearity and complexity of the iterative function, the theoretical analysis is highly nontrivial and technical. We discover a pivotal inequality for delineating the convexity and concavity of the iterative function and establish general auxiliary theories to guarantee the PCP property and convergence. We also develop theories to determine a computable initial guess within a theoretical ``safe" interval. Intriguingly, we find that the unique positive root of a cubic polynomial always falls within this ``safe" interval. To enhance efficiency, we propose a hybrid strategy that combines this with a more cost-effective initial value. The presented PCP NR method is versatile and can be seamlessly integrated into any RMHD numerical scheme that requires the recovery of primitive variables, potentially leading to a very broad impact in this field. As an application, we incorporate it into a discontinuous Galerkin method, resulting in fully PCP schemes. Several numerical experiments, including random tests and simulations of ultra-relativistic jet and blast problems, demonstrate the notable efficiency and robustness of the PCP NR method.
\end{abstract}





\section{Introduction}


In the vast universe, over 99\% of visible matter is in a plasma state. The movement of plasma is governed by magnetohydrodynamics (MHD) equations. Relativistic MHD (RMHD) combines MHD with Einstein's theory of relativity, playing a crucial role in astrophysics and plasma physics, notably in the historic discovery of gravitational waves \cite{ruiz2018gw170817}. RMHD explains plasma behavior in magnetic fields under conditions of near-light speed and/or strong gravitation. It is essential for understanding and predicting phenomena like pulsars, active galactic nuclei, gamma-ray bursts, gravitational waves, jets, and the dynamics near black holes and neutron stars. 


The governing equations for three-dimensional special RMHD can be expressed as a conservative system of hyperbolic conservation laws: 
\begin{equation}
\label{eq:intro_rmhd}
		\frac{\partial \mathbf{U}}{\partial t}+\sum_{i=1}^{3}\frac{\partial \mathbf{F}_i\left(\mathbf{U}\right)}{\partial x_i}={\bf 0},
\end{equation}
accompanied by an additional divergence-free constraint $\nabla \cdot {\mathbf B} = 0$ on the magnetic field $\mathbf B=(B_1,B_2,B_3)^\top$. In \eqref{eq:intro_rmhd}, 
the conservative vector is defined as 
	$\mathbf{U}=\left(D,\mathbf{m}^\top,\mathbf{B}^\top,E\right)^\top$, 
where $D$ represents the mass density, $\mathbf{m} = (m_1, m_2, m_3)^\top$ is the momentum density vector, and $E$ denotes the energy density. The flux vectors in \eqref{eq:intro_rmhd} are given by
\begin{equation}\label{eq:intro_flucvec}
	\mathbf{F}_i\left(\mathbf{U}\right)=\left(Dv_i,v_i\mathbf{m}^\top-B_i(W^{-2}\mathbf{B}^\top+(\mathbf{v}\cdot\mathbf{B})\mathbf{v}^\top)+p_{tot}\mathbf e_i,v_i\mathbf{B}^\top-B_i\mathbf{v}^\top,m_i\right)^\top,
\end{equation}
where $\mathbf{e}_i$ represents the $i$th row of the unit matrix of size 3, $\mathbf{v} = (v_1, v_2, v_3)^\top$ is the fluid velocity vector, $\rho$ denotes the rest-mass density, and $p_{tot} = p + \frac{1}{2}(W^{-2}|\mathbf{B}|^2 + (\mathbf{v} \cdot \mathbf{B})^2)$ signifies the total pressure with $p$ as the thermal pressure. The Lorentz factor $W = (1 - |\mathbf{v}|^2)^{-\frac{1}{2}}$, and the velocity is normalized so that the speed of light  equals one. 

 
Define 
	$\mathbf Q=\left(\rho,\mathbf v^\top,\mathbf B^\top,p\right)^\top$ 
as the primitive variables. 
The conservative vector $\mathbf U$ is explicitly derivable from $\mathbf Q$ through the relations:
\begin{equation}\label{express:Prim2Conserv}
	\begin{cases}
		D=\rho W,\\
		\mathbf m=\rho h W^2\mathbf{v}+|\mathbf B|^2\mathbf{v}-(\mathbf v\cdot \mathbf B)\mathbf B,\\
		E=\rho h W^2-p_{tot}+|\mathbf B|^2,\\
	\end{cases}
\end{equation}
where the specific enthalpy $h$ is determined by $\rho$ and $p$ via an equation of state (EOS): 
\begin{equation}\label{express:enthalpy}
	h = {\mathcal H}(\rho,p),
\end{equation}
for example, the $\gamma$-law EOS with the adiabatic index $\gamma\in(1,2]$ is given by 
\begin{equation}\label{eq:gamma-EOS}
	h={\mathcal H}(\rho,p)=1+{\gamma p}/{((\gamma-1)\rho)}. 
\end{equation}
However, the (inverse) calculation of $\mathbf Q$ from $\mathbf U$ is challenging, as $\mathbf Q$ {\em cannot be explicitly expressed} by $\mathbf U$ due to the strong nonlinear coupling in \eqref{express:Prim2Conserv}.

The complex and nonlinear nature of the RMHD equations \eqref{eq:intro_rmhd} demands the development of advanced numerical schemes for effective RMHD studies. 
Unfortunately, the flux $\mathbf F_i(\mathbf U)$, evaluated at each time step in all computational cells, is a highly nonlinear {\em implicit} function of ${\bf U}$, preventing its direct calculation from $\mathbf U$.  
Therefore, it is imperative to determine the primitive variables $\mathbf{Q}$ from $\mathbf{U}$ before $\mathbf{F}_i$ can be computed. 
{\em A common and intricate challenge faced by all conservative RMHD schemes is to recover  $\mathbf Q$ from $\mathbf U$.}
This procedure is highly complicated, due to the absence of an explicit, closed-form expression.  
Specifically, given a known $\mathbf{U}=\left(D,\mathbf{m}^\top,\mathbf{B}^\top,E\right)^\top$, the core goal is to solve the nonlinear algebraic system  \eqref{express:Prim2Conserv} for the five unknown primitive variables $\{\rho, p, v_1,v_2,v_3\}$, while adhering to the physical constraints: 
\begin{equation}\label{eq:Qphysical}
	\rho>0,~~p>0,~~|\mathbf{v}|=\sqrt{v_1^2+v_2^2+v_3^2}<1.
\end{equation}
As highlighted by Charles Gammie et al.~in \cite{noble2006primitive,gammie2003harm}, an ideal solver for primitive variable recovery should be ``robust, accurate, and fast---it is at the heart of all conservative RMHD schemes''.  
{\em Yet, for over three decades,  the quest for such solvers has been a significant, ongoing challenge.} 
While several solvers are available, ``{\em none has proven completely reliable},'' as pointed out in \cite{kastaun2021robust}. 
Furthermore, the theory on the stability and convergence of primitive variable solvers remains very limited. As mentioned in \cite{siegel2018recovery}, the primitive variable recovery problem {\em``still remains a major source of error, failure, and inefficiency"} in RMHD simulations.
 

\subsection{Related work}
 In the past thirty years, the development of solvers for primitive variables in RMHD has attracted considerable attention. This has led to a variety of iterative schemes, as detailed in, e.g.,  \cite{10.1046/j.1365-8711.1999.02244.x,balsara2001total,gammie2003harm,del2003efficient,noble2006primitive,anton2006numerical,mignone2006hllc,mignone2007equation,cerda2008new,newman2014primitive,palenzuela2015effects} and a systematic review in  \cite{siegel2018recovery}. 
Despite these advancements, many solvers still struggle with inaccuracy, instability, divergence, failure, and/or non-compliance with constraints \eqref{eq:Qphysical}, particularly in the ultra-relativistic and/or strongly magnetized scenarios \cite{siegel2018recovery}.


The Newton--Raphson (NR) method, a renowned iterative technique for solving algebraic equations, is often employed for primitive variable recovery.  
A straightforward strategy is to apply the NR method directly to solve  \eqref{express:Prim2Conserv} for the five unknown variables $\{\rho, p, {\bf v}\}$. 
This approach, termed the ``5D-NR'' method in literature, was first used in \cite{balsara2001total,gammie2003harm}. 
However, studies \cite{del2003efficient,noble2006primitive,siegel2018recovery} have noted the 5D-NR method is slow, inaccurate, and often unstable.  
To develop more efficient solvers, researchers have explored dimensionality reduction of \eqref{express:Prim2Conserv}; see, e.g,  \cite{10.1046/j.1365-8711.1999.02244.x,del2003efficient,noble2006primitive,anton2006numerical,mignone2006hllc,mignone2007equation,cerda2008new,palenzuela2015effects}.  
The primary idea is to carefully reformulate the five equations \eqref{express:Prim2Conserv} into a system of fewer equations with fewer unknown intermediate variables. 
These intermediate variables, once computed by a root-finding algorithm like the NR method for the reduced system, are then used to explicitly calculate all the primitive variables. 
Komissarov \cite{10.1046/j.1365-8711.1999.02244.x} initially reduced  \eqref{express:Prim2Conserv} into a  three-equation system involving three unknowns ($p$, $W$, and $\mathbf v\cdot\mathbf B$), solvable by a 3D-NR method.  
Ant{\'o}n et al.~\cite{anton2006numerical} observed that the system \eqref{express:Prim2Conserv} can be reduced to 
\begin{equation}\label{express:Prim2Conserv_2D}
	\begin{cases}
		|{\bf m}|^2 = (\xi +|{\bf B}|^2)^2 \frac{W^2-1}{W^2} - (2 \xi +|{\bf B}|^2) \frac{ ({\bf m}\cdot {\bf B})^2 }{\xi^2}, \\
		E= \xi + |{\bf B}|^2 - p - \frac{ |{\bf B}|^2 }{2W^2} - \frac{ ({\bf m}\cdot {\bf B})^2 }{\xi^2},
		\\
		D = \rho W, 
\end{cases}
\end{equation}
which can be solved by a 3D-NR method for the intermediate variables $\rho$, $p$, and $W$. Here $\xi = \rho h W^2$. 
Researchers 
further reduced the dimensionality to obtain two-equation systems for different unknown intermediate variables, leading to various 2D-NR methods. 
For example, Noble et al.~\cite{noble2006primitive} derived two equations for 
the intermediate variables $|{\bf v}|^2$ and $\xi$. 
Giacomazzo and Rezzolla 
\cite{giacomazzo2007whiskymhd} noticed $\xi = \rho h (\rho,p) W^2 = \frac{D}{W} h( \frac{D}{W}, p ) W^2$, enabling the first two equations in \eqref{express:Prim2Conserv_2D}  to form a two-equation system for  $p$ and $W$. 
Cerd\'a-Dur\'an et al.~\cite{cerda2008new} reformulated \eqref{express:Prim2Conserv_2D} into two equations for $\xi$ and $W$. 

In \cite{noble2006primitive}, Noble et al.~proposed a 1D-NR method by deriving an implicit equation for the 
intermediate variable $\xi$, where the inversion of EOS was required. 
Mignone et al.~\cite{mignone2006hllc,mignone2007equation} introduced a similar 1D-NR method for solving the following nonlinear equation (taking the $\gamma$-law EOS as example) for the unknown $\xi$: 
\begin{equation}\label{equ:fu_exa}
	{\mathcal F}(\xi):=\xi-\frac{\gamma-1}{\gamma}\left(\frac{\xi}{\mathcal W^2}-\frac{D}{\mathcal W}\right)+|{\bf B}|^2-\frac{1}{2}\left(\frac{|{\bf B}|^2}{\mathcal W^2}+\frac{({\bf m}\cdot {\bf B})^2}{\xi^2}\right)-E=0,
\end{equation}
where $\mathcal W$ is a function of $\xi$ defined by 
\begin{equation}\label{eq:DefWxi}
	\mathcal W(\xi)=\frac{\xi(\xi+|{\bf B}|^2)}{f_{a}(\xi)^{\frac12}} \quad \mbox{with}\quad  f_{a}(\xi):=\xi^2(\xi+|{\bf B}|^2)^2-(\xi^2 |{\bf m}|^2+(2\xi+|{\bf B}|^2)({\bf m}\cdot {\bf B})^2). 
\end{equation}
Once $\xi$ is determined, the primitive variables can then be explicitly calculated. 

Although the recovery of primitive variables can be reduced to a 1D root-finding problem, designing  
a robust 1D-NR solver remains challenging due to the high complexity of the resulting nonlinear equation, as illustrated for the $\gamma$-law EOS in \eqref{equ:fu_exa}. 
The stability and convergence of NR method heavily depend on the initial guess. 
As noted in \cite{siegel2018recovery}, all existing NR solvers for primitive variables ``do not guarantee convergence''. 
Typically, these solvers fail in scenarios involving strongly magnetized fluids or large Lorentz factors \cite{siegel2018recovery}. 
Figure \ref{fig:fu_exa} illustrates a potential failure case where an improperly chosen initial guess $\xi_0$  leads to failure, and our numerical evidence indicates that such failures can even occur with initial guesses close to the true root.

\begin{figure}[!t]
	\centering
	\includegraphics[width=0.5\textwidth]{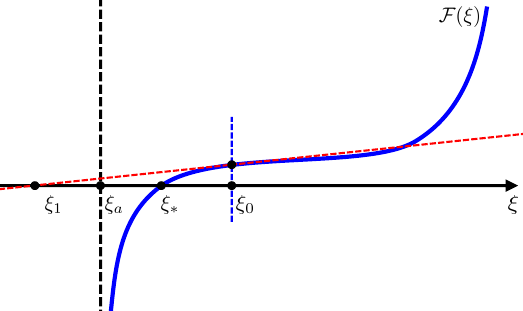}
	\caption{A potential failure scenario in solving \eqref{equ:fu_exa} using the NR method. When $\xi_1 \le \xi_a$, $f_a(\xi_1)\le 0$ so that the values of $\mathcal W(\xi_1)$ and ${\mathcal F}(\xi_1)$ become complex or ill-defined, resulting in  failure. Here $\xi_a$ is the largest non-negative root of $f_a(\xi)$.}\label{fig:fu_exa}
\end{figure}


Yet, finding a unified approach to determine initial guesses that consistently guarantee the robustness and convergence of NR primitive variable solvers remains an open problem. 
A prevalent strategy in the literature is to take the initial guess from the previous time step in numerical evolution  \cite{gammie2003harm,cerda2008new,noble2006primitive,siegel2018recovery}.  
However, its effectiveness is unpredictable and challenging to evaluate \cite{kastaun2021robust}, as the initial guess is not chosen in a deterministic way.  
While this strategy might provide a good initial estimate in cases of smooth solutions and small time step size, it can lead to significant discrepancies from the true root if the time step is not adequately small or the solution exhibits discontinuities. 
The discrepancies can result in divergence or failure, necessitating manual intervention, e.g,  reducing the time step size \cite{newman2014primitive}. Mignone et al.~\cite{mignone2007equation} proposed a more robust initial guess for their 1D-NR method, but a comprehensive theoretical analysis of stability and convergence for this and other NR algorithms remains absent.

In addition to NR-based solvers, several alternative methods were developed for recovering primitive variables in RMHD. For instance, Newman et al.~\cite{newman2014primitive} introduced a fixed-point iteration method with Aitken acceleration for recovering the pressure $p$. 
However, as mentioned in \cite{newman2014primitive}, the expressions involved this method ``are too complicated to provide algebraic proofs of convergence''. 
Another category of solvers utilizes Brent's method \cite{neilsen2014magnetized, palenzuela2015effects}, which combines root bracketing, bisection, and inverse quadratic interpolation. The convergence order of Brent's method switches between 1 (the order of bisection) and 1.839 (the order of inverse quadratic interpolation), both lower than the quadratic convergence order of NR method. Recently, Kastaun et al.~\cite{kastaun2021robust} reformulated the recovery of primitive variables into finding the root of a master function, proved its existence and uniqueness, and employed a robust Brent-type algorithm.
As evidenced in \cite{siegel2018recovery}, while the NR method typically offers greater efficiency and higher accuracy, it is less robust compared to fixed point or Brent's method. To balance efficiency and robustness, Siegel et al.~\cite{siegel2018recovery} suggested starting with an efficient and accurate NR method and, upon its failure, switching to a more robust yet slower algorithm, such as fixed-point iteration \cite{newman2014primitive} or Brent's method \cite{neilsen2014magnetized, palenzuela2015effects}. 
Beyond efficiency, accuracy, and convergence, it is crucial to ensure that the recovered primitive variables  conform to physical constraints \eqref{eq:Qphysical}. Regrettably, many existing primitive variable solvers in RMHD do not always adhere to these constraints. The resulting nonphysical primitive variables make the discrete problems ill-posed and lead to the failure of simulations.  
In recent years, a series of physical-constraint-preserving (PCP) numerical schemes were  proposed for relativistic hydrodynamics  \cite{2015High,wu2017design,WuTang2017ApJS,wu2021minimum,chen2022physical} and RMHD \cite{wu2017admissible,wu2018physical,wu2021provably,wu2021geometric}.  
These PCP schemes, developed via the geometric quasilinearization approach \cite{wu2021geometric}, have been rigorously proven to ensure that the computed conservative variables are physically admissible and comply with the relevant physical constraints. 
 While the PCP property on conservative variables implies the existence and uniqueness of the corresponding physical primitive variables in theory \cite{wu2017admissible, wu2018physical, wu2021provably}, it  however  does not ensure the convergence of primitive variable solvers.  Moreover, the PCP property on conservative variables alone does not necessarily guarantee that the physical constraints \eqref{eq:Qphysical} are met by the numerically computed primitive variables. 
 This highlights a critical gap in current methodologies and underscores the need for further research of fully PCP schemes for robust RMHD simulations. 


\subsection{Contributions and innovations of this paper}
This paper aims to analyze and develop a robust, PCP\footnote{An iterative primitive variable solver is termed PCP if its approximate primitive variables always satisfy the physical constraints \eqref{eq:Qphysical} throughout the iterations.}, and provably convergent NR primitive variable solver. The contributions, innovations, and significance of this work include:

\begin{itemize}[leftmargin=*]
	\item We introduce a robust and efficient NR method for RMHD, building on the 1D-NR methods by Noble et al.~\cite{noble2006primitive} and Mignone et al.~\cite{mignone2006hllc}. Our key innovation is a unified approach for the initial guess, designed based on systematic theoretical analysis. Our design ensures that the NR iteration provably converges and consistently adheres to physical constraints \eqref{eq:Qphysical} throughout the iteration process. To our knowledge, this may be the first proven convergent PCP NR method for RMHD.

	\item The PCP NR method boasts a rapid, provably quadratic convergence rate, due to the iterative function being proven strictly increasing and thus always having a unique (non-repeated) positive root. Empirical evidence shows the NR method's swift convergence to near machine accuracy within five iterations on average.


	\item We establish rigorous mathematical theories to analyze the convergence and stability of the PCP NR method. 
	We propose three auxiliary lemmas for analyzing the convergence of NR method in a generic nonlinear equation, forming the foundation of our analysis. 
	We construct a crucial inequality, which is essential for analyzing the convexity and concavity structure of the iterative function and for proving the PCP property and convergence of our NR method for the $\gamma$-law EOS \eqref{eq:gamma-EOS}. 
	Our theories establish a ``safe" interval for the initial guess that consistently ensures the provable convergence and PCP property of the NR method.

	\item We further derive theories for determining a computable initial value within the theoretical ``safe" interval. Notably, we discover that the unique positive root of a cubic polynomial always lies within the ``safe" interval, and derive a real analytical expression for this root. We propose a hybrid strategy for further enhancing the efficiency.  While our primary focus is the PCP NR method, our theoretical findings extend beyond this specific approach and can broadly apply to the development of other PCP convergent solvers, such as the bisection or Brent's algorithms, for robust recovery of RMHD primitive variables.

\item The PCP NR method is versatile and can be seamlessly integrated into any RMHD numerical scheme that requires the recovery of primitive variables. As an application, we have successfully integrated it into the PCP discontinuous Galerkin (DG) schemes from \cite{wu2017admissible,wu2021provably}, which were proven to maintain the physicality of the computed conservative variables. The PCP NR method guarantees that the recovered primitive variables are always physical. This integration leads to the fully PCP schemes, ensuring all the computational processes in RMHD adhere to physical constraints \eqref{eq:Qphysical}. 
 We implement the PCP NR method and integrated PCP NR-DG schemes, and conduct extensive numerical experiments to demonstrate the notable efficiency and robustness of the PCP NR method compared to six other primitive variable solvers.

\end{itemize}

The paper is structured as follows: Section \ref{sec:whole algor} introduces the PCP convergent NR method, Section \ref{sec:primitive variables recovering algorithms} analyzes its convergence and PCP properties, Section \ref{sec:numerical_tests} provides numerical tests, and Section \ref{sec:conclusion} summarizes the concluding remarks. 


\section{PCP Newton--Raphson Method}\label{sec:whole algor}

This section presents an efficient and highly robust NR method for determining primitive variables in RMHD. The method builds upon the 1D-NR methods of Noble et al.~\cite{noble2006primitive} and Mignone et al.~\cite{mignone2006hllc}. Our key innovation is a unified approach for the initial guess, which is carefully designed based on sophisticated theoretical analysis. This design ensures that the resulting NR iteration is provably convergent and consistently maintains the physical constraints \eqref{eq:Qphysical} throughout the entire NR iteration process. The rigorous proofs of convergence and stability of the NR method are highly technical and will be discussed in Section \ref{sec:primitive variables recovering algorithms}.

In this section, we consider a general causal EOS \eqref{express:enthalpy} satisfying 
\begin{equation}\label{eq:EOScond}
	\begin{cases}
		\mbox{The function ${\mathcal H}(\rho,p)$ in \eqref{express:enthalpy} is differentiable in $\mathbb R^+\times \mathbb R^+$}, &
		\\
		{\mathcal H}(\rho,p) \ge \sqrt{1+p^2/\rho^2}+p/\rho \qquad \qquad \quad~ \mbox{for all } p, \rho>0,
		\\
		{\mathcal H}(\rho,p) \left(\frac1{\rho} - \frac{\partial {\mathcal H}(\rho,p)}{\partial p} \right) < \frac{\partial {\mathcal H}(\rho,p)}{\partial \rho} < 0 \qquad \mbox{for all } p, \rho>0,
		\\
		\lim\limits_{p\to 0^+} {\mathcal H}(\rho,p)  = 1  \qquad  \mbox{for all } \rho>0, 
	\end{cases}
\end{equation}
where the second condition is revealed by relativistic kinetic theory, while the third arises from relativistic causality and the premise that the coefficient of thermal expansion for fluids is positive \cite{WuTang2017ApJS}. 
These assumptions are reasonable and valid for most compressible fluids including gases.  
In particular, they are satisfied by the $\gamma$-law EOS \eqref{eq:gamma-EOS} and several other commonly used EOSs, as shown in \cite{WuTang2017ApJS,wu2018physical}.

\subsection{Physically admissible states}

\begin{Def}
	A conservative vector $\mathbf{U}$ is termed physically admissible if it corresponds to a unique primitive vector $\mathbf{Q} = \left(\rho, \mathbf{v}^\top, \mathbf{B}^\top, p\right)^\top$ that satisfies constraints \eqref{eq:Qphysical}.
\end{Def}

Under the assumptions \eqref{eq:EOScond}, Wu and Tang \cite{wu2017admissible,wu2018physical} derived a sufficient and necessary condition for assessing the admissibility of conservative variables; see Theorem \ref{Thm:WuTang}.

\begin{theorem}[\cite{wu2017admissible,wu2018physical}]\label{Thm:WuTang}
	Consider a general EOS satisfying \eqref{eq:EOScond}. 
	A conservative vector $\mathbf{U}$ is physically admissible, if and only if $\mathbf{U}$ belongs to the following set 
	\begin{equation}\label{express:Uphysical}
		\mathcal{G}=\left\{ \mathbf{U}=\left(D,\mathbf{m}^\top,\mathbf{B}^\top,E\right)^\top:~   D>0,~E-\sqrt{D^2+|{\bf m}|^2}>0,~\Psi(\mathbf U)>0\right\},
	\end{equation}
	where
	\begin{align*}
		\Psi(\mathbf U) &:=(\Phi(\mathbf U)-2(|{\bf B}|^2-E))\sqrt{\Phi(\mathbf U)+|{\bf B}|^2-E}-\sqrt{\frac{27}{2}(D^2|{\bf B}|^2+ ({\bf m}\cdot {\bf B})^2)},
		\\
		\Phi(\mathbf U) &:=\sqrt{(|{\bf B}|^2-E)^2+3(E^2-D^2-|{\bf m}|^2)}. 
	\end{align*}
\end{theorem}

Ensuring that the computed conservative variables are physically admissible is an important, albeit separate, topic; see the PCP numerical schemes developed in \cite{wu2017admissible,wu2018physical,wu2021provably}. Throughout this paper, we always assume that the given conservative vectors are physically admissible and only focus on the NR method that accurately and robustly recovers the corresponding primitive variables.

\subsection{PCP NR method}
Assume that the EOS $h={\mathcal H}(\rho,p)$ can be rewritten as $p={\mathcal P}(\rho,h)$. 
Given a physically admissible state $\mathbf{U}=\left(D,\mathbf{m}^\top,\mathbf{B}^\top,E\right)^\top \in {\mathcal G}$, 
our objective is to recover the associated primitive variables $\mathbf Q=\left(\rho,\mathbf v^\top,\mathbf B^\top,p\right)^\top$. 
Define 
\begin{align}\label{key331}
	&m := |{\bf m}|, \quad B := |{\bf B}|, \quad 
	\tau := \mathbf m\cdot\mathbf B, \quad 
	\alpha_1:=B^2-E,\quad \alpha_2=B^2-m,
	\\ \label{key331b}
	&	\eta=\xi+B^2, \quad \beta_1=		
	\begin{cases}
		\tau^2/B^2,&\text{if}~B\neq0,\\
		0,&\text{if}~B=0, 
	\end{cases} \quad \beta_2=m^2-\beta_1.
\end{align}
Motivated by \cite{noble2006primitive,mignone2006hllc}, we first compute the unknown intermediate variable  $\xi := \rho h W^2$ by finding the positive root of the nonlinear function
\begin{equation}\label{equ:fu general0}
	{\mathcal F}(\xi) :=\xi-{\mathcal P}\left(\frac{D}{\mathcal W},\frac{\xi}{D\mathcal W}\right)-\frac{1}{2}\left(\frac{B^2}{\mathcal W^2}+\frac{\tau^2}{\xi^2}\right) +\alpha_1,   
\end{equation} 
where $\mathcal W(\xi)$ is defined in \eqref{eq:DefWxi}. For the computational stability in ultra-relativistic and strongly magnetized cases, we carefully reformulate $\mathcal W(\xi)$ as 
\begin{equation}\label{fu:express2}
		\mathcal W(\xi)=\left({\frac{(\xi+\alpha_2)(\eta+m)}{\eta^2}+\beta_1\left(\frac{1}{\eta^2}-\frac{1}{\xi^2}\right)}\right)^{-\frac12}. 
\end{equation}
The reformulation \eqref{fu:express2} helps mitigate the effect of 
round-off errors in large-scale cases, as detailed in Appendix \ref{sec:detail}. 
The derivative of ${\mathcal F}(\xi)$ is given by 
\begin{equation}
	{\mathcal F}'(\xi)=1+B^2\varphi_a+\frac{\tau^2}{\xi^3}+ {\mathcal P}_\rho    \left(\frac{D}{\mathcal W},\frac{\xi}{D\mathcal W}\right)D\mathcal W \varphi_a+\frac{1}{D} {\mathcal P}_h \left(\frac{D}{\mathcal W},\frac{\xi}{D\mathcal W}\right)\left(\xi  \mathcal W \varphi_a-\frac{1}{\mathcal W}\right),
\end{equation}
where 
$\mathcal W$ is calculated by \eqref{fu:express2}, $\varphi_a=-\left(\frac{\beta_1}{\xi^3}+\frac{\beta_2}{\eta^3}\right)$, and 
${\mathcal P}_\rho $ and ${\mathcal P}_h$ denote the partial derivatives of ${\mathcal P}(\rho,h)$ with respect to $\rho$ and $h$, respectively.

Our robust PCP NR method for computing $\xi$ proceeds as follows: 
\begin{equation}\label{eq:NR}
	\xi_{n+1} = \xi_n - \frac{ {\mathcal F}(\xi_n) }{ {\mathcal F}'(\xi_n) }, \qquad n=0,1,2,\dots
\end{equation}
with the initial guess defined as 
\begin{equation}\label{eq:initial}
	\xi_0 = \begin{cases}
		\xi_d, & {\mathcal F}(\xi_d) \le 0,
		\\
		\xi_c, & \mbox{otherwise},
	\end{cases}
\end{equation}
where 
\begin{equation}\label{eq:xi_d}
	\xi_d:=\frac{1}{3} \Big( \Phi({\bf U}) - 2 (B^2 - E) \Big) = \frac{1}3 \Big( \sqrt{\alpha_1^2+3(E^2-(D^2+m^2))}-2\alpha_1 \Big),
\end{equation}
and $\xi_c$ is the unique positive root (see Lemma \ref{lem:f3oneroot}) of  
the cubic polynomial 
\begin{equation}\label{eq:f3}
	f_c(\xi):=\xi^3+(B^2-E)\xi^2-\frac{B^2D^2+\tau^2}{2}. 
\end{equation}
In our computations, $\xi_c$ is calculated by
			\begin{equation}\label{eq:xi_c}
			\xi_c=
\begin{cases}
	-\frac{\alpha_1}{3}\left(1-2\cos\left(\frac{\theta}{3}-\frac{\pi}{3}\right)\right), &\text{if } \delta>0,\\
	-\frac13{\left(\alpha_1+\sqrt[3]{X_1+X_2}+\sqrt[3]{X_1-X_2}\right)}, &\text{if } \delta \le 0,
\end{cases}
\end{equation}
where $\delta=27a_0+4 \alpha_1^3$, $\theta=\arccos\left(1+\frac{13.5a_0}{\alpha_1^3}\right)$, $a_0=-\frac12(B^2D^2+\tau^2)$, $X_1=\alpha_1^3+13.5a_0$, and $X_2=1.5\sqrt{3a_0 \delta}$. Note that \eqref{eq:xi_c} does not involve any complex numbers. 

After a sufficient number of NR iterations \eqref{eq:NR}, we obtain an accurately approximate value of $\xi$,  then the primitive variables can be sequentially calculated by 
\begin{align*}
	\mathbf v =\frac{\mathbf m+\xi^{-1} \tau \mathbf B}{\xi+B^2},\quad 
	W=\frac{1}{\sqrt{1-|\mathbf v|^2}},\quad 
	\rho =\frac{D}{W},\quad 
	p = {\mathcal P}\left(\frac{D}{W},\frac{\xi}{DW}\right).
\end{align*}
To facilitate easy implementation of the aforementioned PCP NR method by interested users, we have detailed the pseudocode in Algorithm \ref{algor:2} for the $\gamma$-law EOS.

\begin{remark}
	The above PCP NR method exhibits many remarkable properties. 
	First, its convergence is guaranteed, thanks to our carefully devised initial guess \eqref{eq:initial}. 
	We will provide a rigorous proof of the convergence in Section \ref{sec:primitive variables recovering algorithms} for the $\gamma$-law EOS and validate it for various EOSs through extensive tests  in Section \ref{sec:numerical_tests}.
	The initial guess \eqref{eq:initial} also significantly enhances the robustness of the NR method, making it PCP, i.e., ensures adherence to physical constraints \eqref{eq:Qphysical} throughout the NR iterations \eqref{eq:NR}. Our extensive numerical experiments demonstrate 
	 the PCP NR method's rapid convergence to the target accuracy 
	 $\epsilon_{\tt target} = 10^{-14}$, typically within five iterations on average.  
	 This highlights its exceptional efficiency.
	Moreover, 
	our results in Section \ref{sec:numerical_tests} show that in about 81\% of  random test cases, the condition ${\mathcal F}(\xi_d) \le 0$ in \eqref{eq:initial} is met. In such instances, the complex calculation of $\xi_c$ is avoid, 
	 further enhancing its overall efficiency.
\end{remark}

\begin{remark}
	Since the RMHD equations degenerate into the relativistic hydrodynamic (RHD) system when  $\mathbf B=\mathbf 0$, the PCP NR method is also applicable to RHD without magnetic field. In this case, 
	${\mathcal F}(\xi) =\xi-{\mathcal P}\left(\frac{D}{\mathcal W},\frac{\xi}{D\mathcal W}\right)-E$ with 
	with $\mathcal W(\xi)= \xi/ \sqrt{ (\xi +m) (\xi -m) } $, and the initial guess for the NR method \eqref{eq:NR} can be simply taken as $\xi_0=\xi_c=E$, which ensures the convergence and PCP property.
\end{remark}

\section{Theoretical Analysis of Convergence and Stability}\label{sec:primitive variables recovering algorithms}

This section is dedicated to a rigorous analysis of the convergence and the PCP property of the above NR method. Due to the strong nonlinearity and intricacy of the iterative function ${\mathcal F}(\xi)$, this analysis is challenging and technical. 


Recall the definition \eqref{eq:DefWxi} for the function $\mathcal W(\xi)$, which is involved in ${\mathcal F}(\xi)$. 
To make $\mathcal W(\xi)$ and ${\mathcal F}(\xi)$ well-defined, we require 
$\xi\in\Omega_1:=\mathbb{R}^+\cap \left\{\xi|f_a(\xi)>0\right\},$ 
where $f_{a}(\xi)$ is defined in \eqref{eq:DefWxi}. 
Define 
$$
f_b(\xi) :=f_a(\xi)-D^2(\xi+B^2)^2. 
$$
We recall and summarize the following results proven in \cite{wu2017admissible,wu2018physical}.

\begin{theorem}[\cite{wu2017admissible,wu2018physical}]\label{thm:m3As}
Consider a general EOS satisfying \eqref{eq:EOScond}.	If $\mathbf U \in \mathcal{G}$, then
\begin{itemize}[leftmargin=*]
		\item The quartic polynomial 
		 $f_a(\xi)$ has at least one non-negative root. Let $\xi_a$ denote the largest non-negative root of $f_a(\xi)$, then $\Omega_1=(\xi_a,+\infty)$.
		 
		 \item The quartic polynomial 
		 $f_b(\xi)$ has a unique root $\xi_b$ in $\Omega_1$. Furthermore,  $f_b(\xi)>0$ on $(\xi_b,+\infty)$, $f_b(\xi)<0$ on $(\xi_a,\xi_b)$.
		 
		 \item The function ${\mathcal F}(\xi)$ is strictly increasing on $\Omega_1$, namely, 
		 ${\mathcal F}'(\xi)>0$ for all $\xi \in \Omega_1$. Moreover, 
		 ${\mathcal F}(\xi_b)<0$ and $\lim\limits_{\xi\to +\infty}{\mathcal F}(\xi)=+\infty$. 
		 
		 \item ${\mathcal F}(\xi)$ has a unique root on $\Omega_2:=(\xi_b,+\infty)\subseteq \Omega_1$, denoted by $\xi_*$, which equals $\rho h W^2$. 
		 
\end{itemize}
\end{theorem}

For any $\xi \in\Omega_1$, define the following functions
\begin{equation}\label{eq:xicalpri}
	\mathbf v(\xi) :=\frac{\mathbf m+\xi^{-1}\tau \mathbf B}{\xi+B^2},\quad 
	\rho(\xi) :=\frac{D}{\mathcal W(\xi)},\quad  h(\xi):= \frac{\xi}{D\mathcal W(\xi)}, \quad 
	p(\xi) :=  {\mathcal P}\left( \rho(\xi), h(\xi) \right).
\end{equation}
Given $\mathbf U\in\mathcal{G}$, we employ the NR method \eqref{eq:NR} to approximate the root   
$\xi_*$. Subsequently, the corresponding primitive variables are obtained by inserting this approximate value into \eqref{eq:xicalpri}. 
 Ensuring both the convergence of the NR method \eqref{eq:NR} and adherence to the physical constraints  \eqref{eq:Qphysical} is vital.

\begin{Def} 
	For a given $\mathbf U \in \mathcal{G}$, the NR method \eqref{eq:NR} 
	is termed convergent if the iterative sequence $\{\xi_n\}_{n\geq0}$ converges to the physical root $\xi_*$ of ${\mathcal F}(\xi)$.
\end{Def}

\begin{Def}
	For a given $\mathbf U \in \mathcal{G}$, the NR method \eqref{eq:NR}, utilized to solve ${\mathcal F}(\xi) = 0$, is termed physical-constraint-preserving (PCP) if the iterative sequence $\{\xi_n\}_{n\geq0}$ consistently satisfies
	\begin{equation}\label{eq:PCP}
		\xi_n \in \Omega_1, \quad \rho(\xi_n) > 0, \quad p(\xi_n) > 0, \quad |\mathbf v(\xi_n)| < 1, \quad \forall n \ge 0.
	\end{equation}
\end{Def}

Building upon Theorem \ref{thm:m3As}, we arrive at the following theorem.

\begin{theorem}\label{thm:PCPset}
	For a general EOS satisfying \eqref{eq:EOScond}, 
	the NR method \eqref{eq:NR} is PCP, if and only if  $\xi_n \in\Omega_2$ for all $n \ge 0$.
\end{theorem}
\begin{proof}
	Given $\mathbf U\in\mathcal{G}$, we examine each of the constraints in \eqref{eq:PCP} individually. 
	\begin{itemize}[leftmargin=*]
		\item The constraint $\rho(\xi_n)=\frac{D}{\mathcal W(\xi_n)}>0$ is equivalent to $\mathcal W(\xi_n)>0$. This means $f_a(\xi_n)>0$ and $\xi_n>0$, namely, $\xi_n\in\Omega_1$, according to Theorem \ref{thm:m3As}. 
		\item The second condition in \eqref{eq:EOScond} implies 
		$$p >0 \implies h={\mathcal H}(\rho,p) > 1, \quad \mbox{if } \rho >0.$$ 
		On the other hand, 
		as proven in \cite[Section 2.1]{wu2018physical}, 
		the third and fourth conditions in \eqref{eq:EOScond} imply 
		 ${\mathcal P}_h(\rho,h) > 0$ and $\lim\limits_{h\to 1^+} {\mathcal P}(\rho,h)  = 0$.  
		This means  $ h > 1 \implies p(\rho,h)> 0$. 
		Therefore, 
		the constraint $p(\xi_n)>0$ is equivalent to $h(\xi_n) > 1 $, namely, $\xi_n > D \mathcal W(\xi_n)$, or $f_b(\xi_n)>0$. Thanks to Theorem \ref{thm:m3As}, $p(\xi_n)>0$ is equivalent to $\xi_n \in\Omega_2$. 
		\item Note that 
				\begin{equation*}
				|\mathbf v(\xi_n)|^2-1 
				=\frac{\xi_n^2m^2+\tau^2B^2+2\tau^2\xi_n}{(\xi_n+B^2)^2\xi_n^2}-1
				=-\frac{f_a(\xi_n)}{(\xi_n+B^2)^2\xi_n^2}.
		\end{equation*}
		Thus, $|\mathbf v(\xi_n)|<1$ is equivalent to $f_a(\xi_n)>0$ and $(\xi_n+B^2)\xi_n \neq 0$.
	\end{itemize}
Recall that Theorem \ref{thm:m3As} has established $\Omega_2 \subseteq \Omega_1$.
In conclusion, the constraints \eqref{eq:PCP} are equivalent to $\xi_n \in \Omega_2$ for all $n \ge 0$. This completes the proof.
\end{proof}

\subsection{Main theorems on PCP property and convergence}

As well-known, the stability and convergence of the NR method heavily depend on the initial guess $\xi_0$. 
As noted in \cite{siegel2018recovery}, all existing NR solvers with improper initial guess for RMHD did not guarantee convergence and often failed in scenarios involving highly magnetized fluids or large Lorentz factors. 
For large-scale problems, such issues or failures can arise even when the initial guesses are close to the exact root. Thus, identifying an initial guess that ensures both the stability and convergence of the NR method \eqref{eq:NR} is a highly challenging task. 
Furthermore, even when the NR iterations \eqref{eq:NR} eventually converge to the exact root, the intermediate approximate values $\xi_n$ generated may not satisfy the physical constraints \eqref{eq:PCP}. In such cases, the NR iterations might become unstable or yield nonphysical primitive variables if the iterations terminate prematurely. Therefore, finding an initial guess that consistently secures the PCP property throughout the iterations is both critical and very nontrivial. 




We summarize our key findings in the following theorems.

\begin{theorem}\label{thm:initialset}
	Assume that $\mathbf U \in \mathcal{G}$. 
	For the $\gamma$-law EOS \eqref{eq:gamma-EOS}, 
	if the initial guess $\xi_0\in\Omega_3:=(\xi_b,\xi_*]$, then the NR method \eqref{eq:NR}
	is always PCP and convergent. 
\end{theorem}

The discovery and proof of Theorem \ref{thm:initialset} are highly nontrivial and technical. For better readability, we put the proof in Section \ref{sec:cproof}, after establishing some auxiliary theories in Section \ref{sec:auxi} and a crucial inequality in Section \ref{sec:ineq}.

Due to the technical difficulties, Theorem \ref{thm:initialset} is only restricted to the $\gamma$-law EOS. 
The rigorous convergence analysis for general EOSs presents a greater challenge. 
We only have a preliminary result, as shown in Theorem \ref{thm:initialset_gEOS}. 
Some numerical evidence will be presented in Section \ref{sec:numerical_tests} to validate the efficiency and robustness of the PCP NR method for various EOSs.

\begin{theorem}\label{thm:initialset_gEOS}
	Assume that $\mathbf U \in \mathcal{G}$. 
		Consider a general EOS \eqref{express:enthalpy} satisfying \eqref{eq:EOScond} and that  
	${\mathcal F}'(\xi)$ is monotone on $\Omega_2 = (\xi_b, +\infty)$, or there is an inflection point $\xi_{in} \in \Omega_2$ such that ${\mathcal F}'(\xi)$ is monotonically decreasing on $ (\xi_b,\xi_{in}]$ and increasing on $[\xi_{in},+\infty)$. 
	If the initial guess $\xi_0\in\Omega_3:=(\xi_b,\xi_*]$, then the NR method \eqref{eq:NR}
	is PCP and convergent. 
\end{theorem}

The proof of Theorem \ref{thm:initialset_gEOS} is also given in Section \ref{sec:cproof}.

Theorem \ref{thm:initialset} provides a ``safe'' interval $\Omega_3$ for the initial guess $\xi_0$ that consistently guarantees the PCP property and convergence of the NR method \eqref{eq:NR} for the $\gamma$-law EOS. Our numerical evidence supports the assumption in Theorem \ref{thm:initialset_gEOS} (see Figure \ref{fig:Df} for the two patterns of ${\mathcal F}'(\xi)$ observed in extensive random experiments), indicating that $\Omega_3$ should also be a ``safe'' interval for other EOSs. 
However, pinpointing a computable initial value $\xi_0$ within $\Omega_3$ is a new challenge. This is because the right endpoint $\xi_*$ of $\Omega_3$ is the unknown exact root of ${\mathcal F}(\xi)$, while the left endpoint $\xi_b$ is a root of the quartic polynomial $f_b(\xi)$ which is difficult to calculate efficiently. 
Interestingly, our theory presented in Section \ref{sec:initialvalue} reveals that 
the unique positive root (denoted as $\xi_c$) of 
the cubic polynomial $f_c(\xi)$ in \eqref{eq:f3} is consistently located within the ``safe'' interval $\Omega_3$, offering a calculable and stable initial value for the NR method \eqref{eq:NR}. 
On the other hand, we find that the quantity $\xi_d$ defined in \eqref{eq:xi_d}
is easier to compute and may also provide a robust initial guess. 
These findings lead to the following critical theorem.

\begin{theorem}\label{thm:wholeNR} 
	Assume that $\mathbf U \in \mathcal{G}$. 
	Consider the $\gamma$-law EOS \eqref{eq:gamma-EOS} or a general EOS \eqref{express:enthalpy} satisfying the assumptions in Theorem \ref{thm:initialset_gEOS}. 
	If the initial guess $\xi_0$ is chosen as \eqref{eq:initial} or set as $\xi_c$, 
	then the NR method \eqref{eq:NR}
	is PCP and convergent, exhibiting a quadratic convergence rate.
\end{theorem}
\begin{proof}
	The proof is based on Theorems \ref{thm:initialset}--\ref{thm:initialset_gEOS} and the theories on the initial guess presented in Section \ref{sec:initialvalue}. 
	Theorem \ref{thm:f3omega3} confirms that $\xi_c$ always falls within the ``safe'' interval $\Omega_3 = (\xi_b, \xi_*]$. According to Theorems \ref{thm:initialset}--\ref{thm:initialset_gEOS}, choosing $\xi_c$ as the initial guess ensures the PCP property and convergence of the NR method \eqref{eq:NR}. 
	If ${\mathcal F}(\xi_d) \le 0$, the monotonicity of ${\mathcal F}(\xi)$ shown in Theorem \ref{thm:m3As} implies $\xi_d \le \xi_*$, which along with Theorem \ref{lem:xid} yields $\xi_d \in \Omega_3=(\xi_b,\xi_*]$.  
	Consequently, if the initial guess $\xi_0$ is chosen as \eqref{eq:initial}, then the NR method \eqref{eq:NR}
	is also consistently PCP and convergent. 
	
	As stated in Theorem \ref{thm:m3As}, ${\mathcal F}'(\xi)>0$ for all $\xi \in \Omega_2$, indicating that the exact root $\xi_*$ of ${\mathcal F}(\xi)$ is not a repeated root. Hence, the PCP NR method \eqref{eq:NR} converges quadratically. 
 The proof is completed. 
\end{proof}

\begin{figure}[!t]
	\centering
	\includegraphics[width=0.635\textwidth]{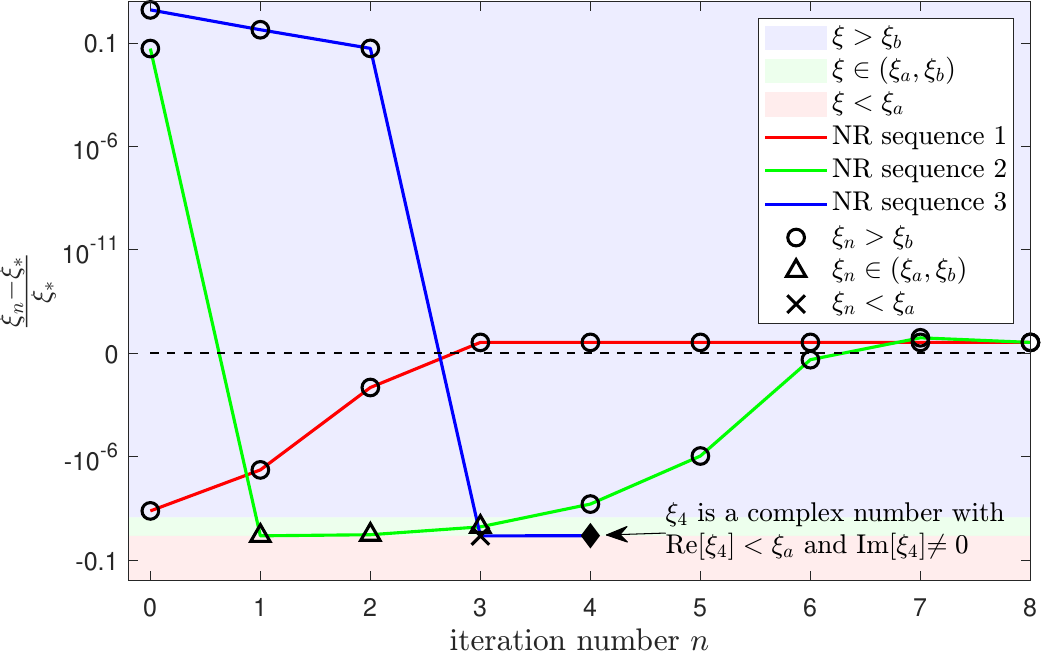} 
	\caption{Three cases of NR iterations.}\label{fig:NR}
\end{figure}

\begin{remark}
	Selecting an arbitrary initial guess within $\Omega_2 = (\xi_b, +\infty)$ may not consistently ensure the PCP property, as demonstrated by NR sequence 2 in Figure \ref{fig:NR}. Furthermore, such a selection can even cause the iterative sequence $\{\xi_n\}_{n \geq 0}$ to go out of $\Omega_1$ and become non-real, as evidenced by the NR sequence 3 in Figure \ref{fig:NR}.
In contrast, our proposed initial guess $\xi_0$, defined in \eqref{eq:initial}, consistently maintains the PCP property and guarantees convergence, as exemplified by the NR sequence 1 in Figure \ref{fig:NR}.
\end{remark}

\subsection{Auxiliary theories and proofs of Theorems \ref{thm:initialset} and \ref{thm:initialset_gEOS}}\label{sec:proof}

We establish several auxiliary results as stepping stones to prove Theorems \ref{thm:initialset} and \ref{thm:initialset_gEOS}.

\subsubsection{Auxiliary theories}\label{sec:auxi} 
We begin by presenting three fundamental lemmas that establish general theories about the convergence of the NR method when applied to a generic nonlinear equation $f(\xi) = 0$. In this context, $f(\xi)$ represents a general function, which includes but is not limited to the special function ${\mathcal F}(\xi)$ discussed in this paper. The exact root of $f(\xi)$ is also denoted by $\xi_*$.


\begin{lemma}\label{lem:monoNRconverges1}
	Let $\{\xi_n\}_{n\geq0}$ denote the iteration sequence obtained using the NR method to 
	compute the root $\xi_*$ of a general function 
	 $f(\xi)$, which is differentiable on $[\xi_0, \xi_*]$. If 
	 one of the following two conditions holds: 
	\begin{itemize}[leftmargin=*]
		\item $f'(\xi)<0$ for all $\xi\in [\xi_0,\xi_{*})$, and 
		$f'(\xi)$ is monotonically increasing on  $[\xi_0,\xi_{*})$; 
		\item $f'(\xi)>0$ for all $\xi\in [\xi_0,\xi_{*})$, and 
		$f'(\xi)$ is monotonically decreasing on  $[\xi_0,\xi_{*})$; 
	\end{itemize}
	then the NR iteration sequence $\{\xi_n\}_{n\geq0}$ is  monotonically increasing and  converges to $\xi_{*}$.
\end{lemma}

\begin{proof}
	The proof follows that of \cite[Lemma 2.4]{cai2023provably} and is thus omitted. 
\end{proof}

\begin{lemma}\label{lem:monoNRconverges2}
	Let $\{\xi_n\}_{n\geq0}$ denote the iteration sequence obtained using the NR method to 
	compute the root $\xi_*$ of a general function 
	$f(\xi)$, which is differentiable on $[\xi_*, \xi_0]$.  If 
	 one of the following two conditions holds:
	\begin{itemize}[leftmargin=*]
		\item $f'(\xi)<0$ for all $\xi\in (\xi_{*},\xi_0]$, and 
		$f'(\xi)$ is monotonically decreasing on  $\xi\in (\xi_{*},\xi_0]$; 
		\item $f'(\xi)>0$ for all $\xi\in (\xi_{*},\xi_0]$, and 
		$f'(\xi)$ is monotonically increasing on  $\xi\in (\xi_{*},\xi_0]$; 
	\end{itemize}
	then the NR iteration sequence $\{\xi_n\}_{n\geq0}$ monotonically decreases and  converges to $\xi_{*}$.
\end{lemma}

\begin{proof}
	The proof follows that of \cite[Lemma 2.5]{cai2023provably} and is thus omitted. 
\end{proof}

%

\begin{figure}[!t]
	\centering
	\subfloat[]{
		\includegraphics[width=0.21\textwidth]{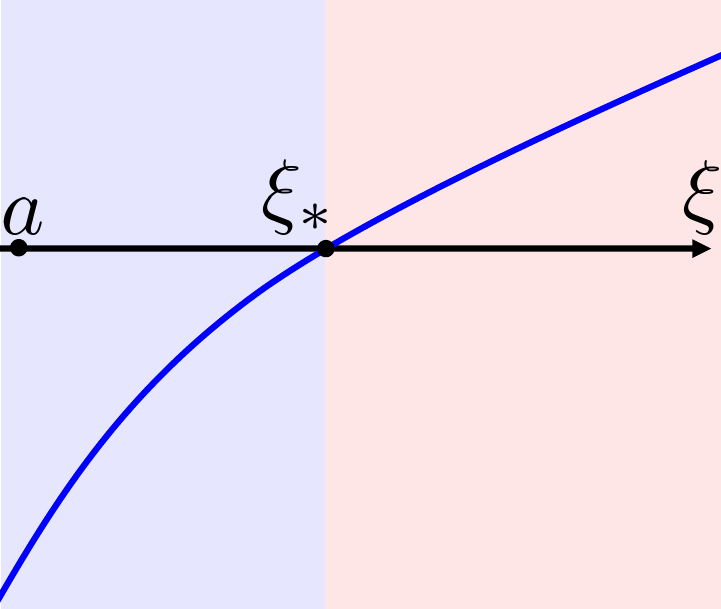}\label{fig:convexity1}
	} ~~~
	\subfloat[]{
		\includegraphics[width=0.21\textwidth]{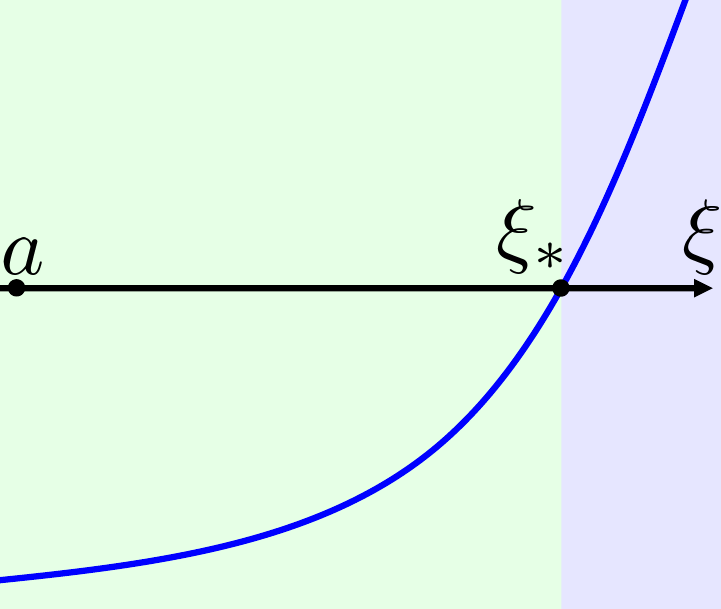}\label{fig:convexity2}
	} ~~~
	\subfloat[]{
		\includegraphics[width=0.21\textwidth]{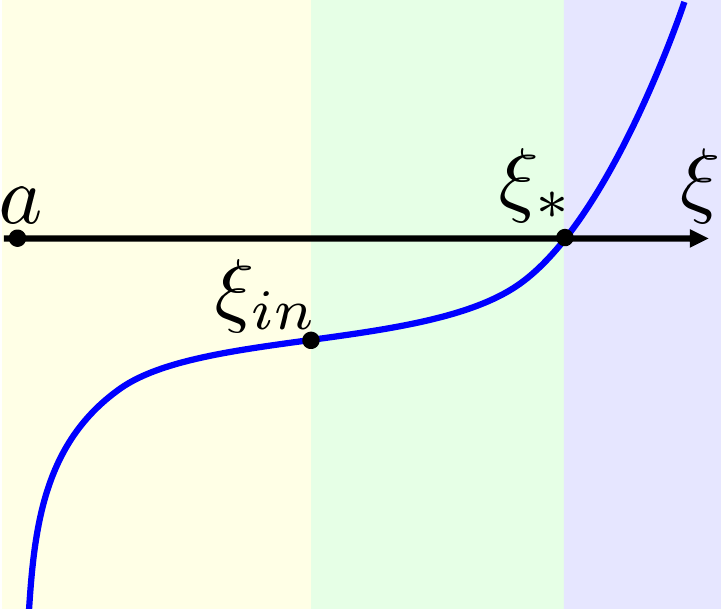}\label{fig:convexity3}
	}~~~
	\subfloat[]{
		\includegraphics[width=0.21\textwidth]{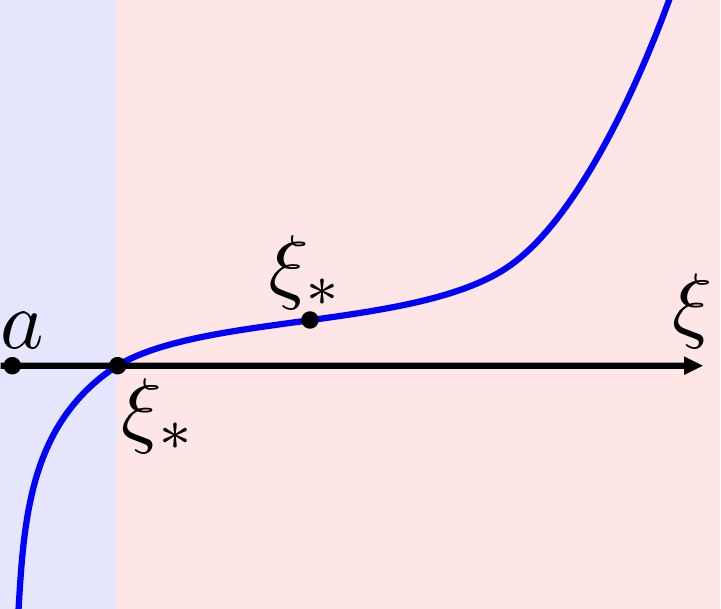}\label{fig:convexity4}
	}
	\caption{Four possible patterns of $f(\xi)$ described in Lemma \ref{lem:NRconverges}.}\label{fig:convexity}
\end{figure}

Now we are ready to prove the following general result. 

\begin{lemma}\label{lem:NRconverges}
	Assume that $f(\xi)$ is a differentiable function on $(a,+\infty)$, $f(a)<0$, $\lim\limits_{\xi\to +\infty}f(\xi)>0$, and $f'(\xi)>0$ for all $\xi \in (a,+\infty)$. Suppose that one of the following two  conditions is satisfied:
	\begin{itemize}[leftmargin=*]
		\item $f'(\xi)$ is monotone on $(a,+\infty)$;
		\item There exists $\xi_{in}\in(a,+\infty)$, such that $f'(\xi)$ is monotonically decreasing on $(a,\xi_{in}]$ and increasing on $[\xi_{in},+\infty)$.
	\end{itemize}
	Then we have: 
	\begin{itemize}[leftmargin=*]
		\item $f(\xi)$ has a unique root $\xi_{*}\in(a,+\infty)$;
		\item For the NR method solving the equation $f(\xi)=0$, if the initial value $\xi_0\in(a,\xi_{*}]$, then the iteration sequence $\left\{\xi_n\right\}_{n\geq0}\subseteq (a,+\infty)$ and $\lim\limits_{n\to +\infty}\xi_n=\xi_{*}$.
	\end{itemize}
\end{lemma}

\begin{proof}
	Since $f'(\xi)>0$ for all $\xi \in (a,+\infty)$, the function $f(\xi)$ is strictly increasing on $(a,+\infty)$. 
	Because $f(a)<0$ and $\lim\limits_{\xi\to +\infty}f(\xi)>0$, we know that $f(\xi)$ has a unique root within $(a,+\infty)$, denoted by $\xi_{*}$. 
	Next, we study the behavior of the NR iteration sequence for solving the equation $f(\xi)=0$ with the initial guess $\xi_0\in(a,\xi_{*}]$. 
		Based on the monotonicity conditions of $f'(\xi)$, we can determine that the concavity/convexity structure of $f(\xi)$ can only fall into one of the four cases depicted in Figure \ref{fig:convexity}. Let us consider each case separately. 
	\begin{enumerate}[label = {(\alph*)},leftmargin=*]
		\item   
		As illustrated in Figure \ref{fig:convexity1}, in this case, $f'(\xi)$ is monotonically decreasing on $(a,+\infty)$. 
		When the initial value $\xi_0\in(a,\xi_{*})$, the NR iteration sequence $\{\xi_n\}_{n\geq0}$ exhibits a monotonic increase and converges to $\xi_{*}$, according to Lemma \ref{lem:monoNRconverges1}. Consequently, $\left\{\xi_n\right\}_{n\geq0}\subseteq (a,+\infty)$ and $\lim\limits_{n\to +\infty}\xi_n=\xi_{*}$.
		
		\item As illustrated in Figure \ref{fig:convexity2}, in this case, $f'(\xi)$ is monotonically increasing on $(a,+\infty)$. If $\xi_0\in(a,\xi_{*}]$, then $\xi_1=\xi_0-\frac{f(\xi_0)}{f'(\xi_0)}$, and $f(\xi_1)\geq f(\xi_0)+(\xi_1-\xi_0)f'(\xi_0)=0$. Thus, $\xi_1\in[\xi_{*},+\infty)$. According to  Lemma \ref{lem:monoNRconverges2}, 
		the NR iteration sequence $\left\{\xi_n\right\}_{n\geq1}$ is monotonically decreasing and converges to $\xi_{*}$. Therefore,  $\left\{\xi_n\right\}_{n\geq0}={\xi_0}\cup\left\{\xi_n\right\}_{n\geq1}\subseteq (a,+\infty)$ and $\lim\limits_{n\to +\infty}\xi_n=\xi_{*}$. 
		
		\item As illustrated in Figure \ref{fig:convexity3}, in this case, there exists an inflection point $\xi_{in}\in(a,\xi_{*})$ such that $f'(\xi)$ is monotonically decreasing on $(a,\xi_{in}]$ and increasing on $[\xi_{in},+\infty)$.

		\begin{enumerate}[label = (\Roman*)]
			\item 		If $\xi_0\in[\xi_{in},\xi_{*}]$, then similar to the above case (b), we have  $\xi_1\in[\xi_{*},+\infty)$. Consequently, the NR iteration sequence $\left\{\xi_n\right\}_{n\geq1}$ is monotonically decreasing and converges to $\xi_{*}$, due to Lemma \ref{lem:monoNRconverges2}. 
			Thus we have $\left\{\xi_n\right\}_{n\geq0}={\xi_0}\cup\left\{\xi_n\right\}_{n\geq1}\subseteq (a,+\infty)$ and $\lim\limits_{n\to +\infty}\xi_n=\xi_{*}$.
			\item  If $\xi_0\in(a,\xi_{in})$, then $\xi_1=\xi_0-\frac{f(\xi_0)}{f'(\xi_0)}>\xi_0$ and we need to investigate the following three situations: 
			
			\begin{itemize}[leftmargin=*]
				\item If $\xi_1>\xi_{*}$, then the NR iteration sequence $\left\{\xi_n\right\}_{n\geq1}$ is  monotonically decreasing and converges to $\xi_{*}$, according to Lemma \ref{lem:monoNRconverges2}, and thus  $\left\{\xi_n\right\}_{n\geq0}={\xi_0}\cup\left\{\xi_n\right\}_{n\geq1}\subseteq (a,+\infty)$. 
				\item If $\xi_{*}\geq\xi_1\geq\xi_{in}$, then the discussion returns to the above case (I), and we have $\left\{\xi_n\right\}_{n\geq1}\subseteq (a,+\infty)$ and $\lim\limits_{n\to +\infty}\xi_n=\xi_{*}$. Consequently,  $\left\{\xi_n\right\}_{n\geq0}={\xi_0}\cup\left\{\xi_n\right\}_{n\geq1}\subseteq (a,+\infty)$. 
				\item If $\xi_1<\xi_{in}$, then $\xi_2=\xi_1-\frac{f(\xi_1)}{f'(\xi_1)}>\xi_1$. 
				By repeatedly following
				the aforementioned discussions, as long as $\xi_m\geq\xi_{in}$ appears in the iteration, we have either $\xi_m>\xi_{*}$ or $\xi_{*}\geq\xi_m\geq\xi_{in}$, for both cases we have  $\left\{\xi_n\right\}_{n\geq0}=\left\{\xi_n\right\}_{m>n\geq0}\cup\left\{\xi_n\right\}_{n\geq m}\subseteq (a,+\infty)$ and $\lim\limits_{n\to +\infty}\xi_n=\xi_{*}$.		
				 			It remains
				 to discuss whether it is possible that $\xi_n < \xi_{in}$ for all $n\geq1$. Assume that such a situation
				 occurs, then by the monotone bounded convergence theorem, the sequence $\left\{\xi_n\right\}_{n\geq0}$ must have a limit $\xi_{**}$, which satisfies $\xi_{**}\leq\xi_{in}<\xi_{*}$. It follows that 
				 $$0=\displaystyle\lim_{n\to +\infty}(\xi_{n+1}-\xi_n)=-\displaystyle\lim_{n\to +\infty}\frac{f(\xi_n)}{f'(\xi_n)}=-\frac{f(\xi_{**})}{f'(\xi_{**})},$$ 
				 which yields $f(\xi_{**})=0$. This is contradictory to $\xi_{**}<\xi_{*}$, since $\xi_{*}$ is the unique root of $f(\xi)$ on $(a,+\infty)$. 
				 Hence, the assumption is incorrect, and there always
				 exists a $m\geq1$ such that $\xi_m\geq\xi_{in}$. 
			\end{itemize}
		\end{enumerate} 
		In short, if $\xi_0\in(a,\xi_*]$, we always have $\left\{\xi_n\right\}_{n\geq0}\subseteq (a,+\infty)$ and $\lim\limits_{n\to +\infty}\xi_n=\xi_{*}$.
		\item 
		As illustrated in Figure \ref{fig:convexity4}, in this case, there exists an inflection point $\xi_{in}\in[\xi_{*},+\infty)$ such that $f'(\xi)$ is monotonically decreasing on $(a,\xi_{in}]$ and increasing on $[\xi_{in},+\infty)$. 
		If $\xi_0\in(a,\xi_{*}]$, then the iteration sequence $\left\{\xi_n\right\}_{n\geq0}$ 
		is monotonically increasing and 
		converges  to $\xi_{*}$, according to Lemma \ref{lem:monoNRconverges1}. Therefore, $\left\{\xi_n\right\}_{n\geq0}\subseteq (a,+\infty)$ and $\lim\limits_{n\to +\infty}\xi_n=\xi_{*}$.
	\end{enumerate}
	The proof is completed. 
\end{proof}

\subsubsection{A crucial inequality}\label{sec:ineq} 

We discover the following important inequality, which will play a key role in understanding the structure of ${\mathcal F}(\xi)$ and proving the PCP property and convergence of our NR method \eqref{eq:NR} for the  $\gamma$-law EOS. We emphasize that this is a very nontrivial discovery.  

\begin{theorem}\label{lem:inequality}
	Consider the $\gamma$-law EOS \eqref{eq:gamma-EOS}. For any $\xi \in \Omega_2$ and $\mathbf U \in \mathcal{G}$, it holds 
	\begin{equation}\label{eq:ineq}
	  \xi	{\mathcal F}'''(\xi)+4{\mathcal F}''(\xi) > 0. 
	\end{equation}
\end{theorem}

\begin{proof}
	Let $g(\xi):=-\frac{\gamma}{\gamma-1} {\mathcal F}''(\xi)$. Because $\gamma >1$,  we only need to prove 
		\begin{equation}\label{eq:ineq_g}
		\xi	g'(\xi)+4 g(\xi) < 0. 
	\end{equation}
	Define
	\begin{equation*}
		\begin{aligned}
			&\varphi_a(\xi) :=-\frac{\tau^2(3\xi^2+3\xi B^2+B^4)+m^2\xi^3}{\xi^3(\xi+B^2)^3},\\
			&\varphi_b(\xi) :=\frac{(2\xi+B^2)\tau^2(2\xi^2+2\xi B^2+B^4)+m^2\xi^4}{\xi^4(\xi+B^2)^4},\\
			&\varphi_c(\xi) :=-\frac{\tau^2(5\xi^4+10\xi^3B^2+10\xi^2B^4+5B^6\xi+B^8)+m^2\xi^5}{\xi^5(\xi+B^2)^5},
		\end{aligned}
	\end{equation*}
	which satisfy 
	\begin{equation*}
		\begin{aligned}
			&\varphi_a'(\xi)=3\varphi_b(\xi),\qquad \varphi_b'(\xi)=4\varphi_c(\xi),\\
			&\mathcal W'(\xi)=\varphi_a(\xi)\mathcal W^3(\xi),\qquad \mathcal W''(\xi)=3\mathcal W^5(\xi)\varphi_a^2(\xi)+3\mathcal W^3(\xi)\varphi_b(\xi).
		\end{aligned}
	\end{equation*}
	Then we have 
	\begin{equation*}
		\begin{aligned}
			&{\mathcal F}'(\xi)=1+B^2\varphi_a+\frac{\tau^2}{\xi^3}-\frac{\gamma-1}{\gamma}\left(\frac{1}{\mathcal W^2}-2\xi \varphi_a+D\mathcal W\varphi_a\right),\\
			&
			\resizebox{1\hsize}{!}{${\mathcal F}''(\xi)=-\left ( \mathcal W^3\frac{\gamma-1}{\gamma}D\varphi_a^2+3\mathcal W\frac{\gamma-1}{\gamma}D\varphi_b+\left(\frac{3\tau^2}{\xi^4}-3\varphi_b\Big(B^2+2\frac{\gamma-1}{\gamma}\xi\Big)-4\frac{\gamma-1}{\gamma}\varphi_a\right)\right).$}
		\end{aligned}
	\end{equation*}
	It follows that 
	\begin{equation*}
		g(\xi)=-\frac{\gamma}{\gamma-1}{\mathcal F}''(\xi)=D\varphi_a^2\mathcal W^3+3D\varphi_b\mathcal W+\frac{3\gamma}{\gamma-1}\frac{\tau^2}{\xi^4}-3\varphi_b(\frac{\gamma}{\gamma-1}B^2+2\xi)-4\varphi_a.
	\end{equation*}
	Direct calculations yield that
		\begin{align*}
			g'(\xi)\xi+4g(\xi)=&3D\xi \varphi_a^3\mathcal W^5+9D\xi \varphi_a\varphi_b\mathcal W^3+4D\varphi_a^2\mathcal W^3+12D\varphi_b\mathcal W+12D\varphi_c\mathcal W\xi\\
			&-12\left(\frac{\gamma}{\gamma-1}B^2+2\xi\right)(\varphi_c\xi+\varphi_b)-16\varphi_a-18\varphi_b\xi.
		\end{align*}
	Noting that $\frac{\gamma}{\gamma-1}\geq2$ and 
	\begin{equation}\label{equ:1}
		\varphi_b(\xi)+\varphi_c(\xi)\xi=\frac{m^2B^2-\tau^2}{(\xi+B^2)^5}>0,
	\end{equation}
	we obtain 
	\begin{equation}\label{WKL1}
		\begin{aligned}
			g'(\xi)\xi+4g(\xi)\leq&3D\xi \varphi_a^3\mathcal W^5+9D\xi \varphi_a\varphi_b\mathcal W^3+4D\varphi_a^2\mathcal W^3+12D\varphi_b\mathcal W+12D\varphi_c\mathcal W\xi\\
			& +\left(-24(B^2+\xi)(\varphi_c\xi+\varphi_b)-16\varphi_a-18\varphi_b\xi \right) \\
		   =&D\mathcal W\left( 3\xi \varphi_a^3\mathcal W^4+9\xi \varphi_a\varphi_b\mathcal W^2+4\varphi_a^2\mathcal W^2+12\varphi_b+12\varphi_c\xi \right)\\
		   &+\left(-24(B^2+\xi)(\varphi_c\xi+\varphi_b)-16\varphi_a-18\varphi_b\xi \right).
		\end{aligned}
	\end{equation}
For the last term in \eqref{WKL1}, using \eqref{equ:1} and 
		$\varphi_a(\xi)+\varphi_b(\xi)\xi=\frac{\tau^2-m^2B^2}{(\xi+B^2)^4}$ 
gives 
	\begin{equation*}
		\begin{aligned}
			&-24(B^2+\xi)(\varphi_c\xi+\varphi_b)-16\varphi_a-18\varphi_b\xi \\
			=&-6\left[4(B^2+\xi)\left(\varphi_c\xi+\varphi_b\right)+3(\varphi_a+\varphi_b\xi)\right]+2\varphi_a\\
			=&-6\frac{m^2B^2-\tau^2}{(\xi+B^2)^4}+2\varphi_a<0,
		\end{aligned}
	\end{equation*}
	where $\varphi_a(\xi)<0$ and $m^2B^2>\tau^2$ have been used.
	Next, we study the upper bound of \eqref{WKL1} in two cases: 
	\begin{itemize}[leftmargin=*]
	\item  Case 1:	$3\xi \varphi_a^3\mathcal W^4+9\xi \varphi_a\varphi_b\mathcal W^2+4\varphi_a^2\mathcal W^2+12\varphi_b+12\varphi_c\xi\leq0$. 
	
		In this case, it is evident that
		\begin{equation*}
			\begin{aligned}
				&D\mathcal W\Big( 3\xi \varphi_a^3\mathcal W^4+9\xi \varphi_a\varphi_b\mathcal W^2+4\varphi_a^2\mathcal W^2+12\varphi_b+12\varphi_c\xi \Big)\\
				&-24(B^2+\xi)(\varphi_c\xi+\varphi_b)-16\varphi_a-18\varphi_b\xi<0.
			\end{aligned}
		\end{equation*}
		
		\item Case 2: 
		$3\xi \varphi_a^3\mathcal W^4+9\xi \varphi_a\varphi_b\mathcal W^2+4\varphi_a^2\mathcal W^2+12\varphi_b+12\varphi_c\xi>0$. 
		
		In this case,  using  
		$\xi\in\Omega_2$, which implies $f_b(\xi)>0$ and $D\mathcal W(\xi)<\xi$, we obtain 
		\begin{equation*}
			\begin{aligned}
				&D\mathcal W \Big( 3\xi \varphi_a^3\mathcal W^4+9\xi \varphi_a\varphi_b\mathcal W^2+4\varphi_a^2\mathcal W^2+12\varphi_b+12\varphi_c\xi \Big)\\
				&-24(B^2+\xi)(\varphi_c\xi+\varphi_b)-16\varphi_a-18\varphi_b\xi\\
			   <&\xi \Big( 3\xi \varphi_a^3\mathcal W^4+9\xi \varphi_a\varphi_b\mathcal W^2+4\varphi_a^2\mathcal W^2+12\varphi_b+12\varphi_c\xi \Big)\\
			   &-24(B^2+\xi)(\varphi_c\xi+\varphi_b)-16\varphi_a-18\varphi_b\xi.
			\end{aligned}
		\end{equation*}
	\end{itemize}
	Combining the estimates from these two cases, we establish that to show \eqref{eq:ineq_g}, 
it is sufficient to prove 
	\begin{equation} \label{WKL2}
		\begin{aligned}
			\phi (\xi):= & \xi\Big( 3\xi \varphi_a^3\mathcal W^4+9\xi \varphi_a\varphi_b\mathcal W^2+4\varphi_a^2\mathcal W^2+12\varphi_b+12\varphi_c\xi\Big)\\
			&-24(B^2+\xi)(\varphi_c\xi+\varphi_b)-16\varphi_a-18\varphi_b\xi\leq0, \quad \forall \xi \in \Omega_2.
		\end{aligned}
	\end{equation}
	Using the notations in \eqref{key331b} and recalling that $\eta>0$ and $\beta_1,\beta_2 \ge 0$, 
we reformulate 
	\begin{equation}\label{subs:2}
		\begin{aligned}
			&\varphi_a=-\left(\frac{\beta_1}{\xi^3}+\frac{\beta_2}{\eta^3}\right), \qquad 
			\varphi_b=\frac{\beta_1}{\xi^4}+\frac{\beta_2}{\eta^4},\\
			&\varphi_c=-\left(\frac{\beta_1}{\xi^5}+\frac{\beta_2}{\eta^5}\right), \qquad \mathcal W=\left( 1-\Big(\frac{\beta_1}{\xi^2}+\frac{\beta_2}{\eta^2}\Big) \right)^{-\frac12}.
		\end{aligned}
	\end{equation}
Substituting \eqref{subs:2} into $\phi(\xi)$ in \eqref{WKL2}, we obtain 
	\begin{equation}\label{ineq:3}
		\begin{aligned}
		 \mathcal W^4(\xi) \phi(\xi) =	& -\frac{2 \beta_1}{\xi^3}-\frac{8 \beta_2}{\eta^3}-\frac{\beta_1^2}{\xi^5}+\frac{16 \beta_2^2}{\eta^5}-\frac{8 \beta_2^3}{\eta^7}+\frac{18 \xi \beta_2}{\eta^4}-\frac{32 \xi \beta_2^2}{\eta^6}+\frac{14 \xi \beta_2^3}{\eta^8}-\frac{6 \xi^2 \beta_2^3}{\eta^9} \\
			& +\frac{15 \xi^2 \beta_2^2}{\eta^7}-\frac{12 \xi^2 \beta_2}{\eta^5}+\frac{15 \beta_1 \beta_2}{\xi^2 \eta^3}-\frac{16 \beta_1^2 \beta_2}{\xi^4 \eta^3}+\frac{32 \beta_1 \beta_2^2}{\xi \eta^6}+\frac{\beta_1^2 \beta_2}{\xi^5 \eta^2}-\frac{45 \beta_1 \beta_2}{\xi \eta^4} \\
			& +\frac{27 \beta_1^2 \beta_2}{\xi^3 \eta^4}-\frac{15 \beta_2^2 \beta_1}{\eta^5 \xi^2}-\frac{15 \beta_2^2 \beta_1}{\eta^7}+\frac{24 \beta_2 \beta_1}{\eta^5}-\frac{12 \beta_2 \beta_1^2}{\eta^5 \xi^2}+\frac{4 \beta_1 \beta_2}{\xi^3 \eta^2}-\frac{2 \beta_1 \beta_2^2}{\xi^3 \eta^4}.
		\end{aligned}
	\end{equation}
	For any $\xi \in \Omega_2$, we have $0\le \frac{\beta_1}{\xi^2}+\frac{\beta_2}{\eta^2}<1$, which inspires us to introduce two auxiliary variables: 
	\begin{equation}\label{subs:3}
		\begin{aligned}
			&P:= \beta_2/\eta^2, \qquad Q:=\beta_1/\xi^2. 
		\end{aligned}
	\end{equation}
	Substituting \eqref{subs:3} into \eqref{ineq:3} gives 
	\begin{equation}\label{ineq:4}
	\begin{aligned}
	\mathcal W^4(\xi) &\phi(\xi) = 	 -\frac{1}{\xi\eta^3}\Big( \left(2QP^2-Q^2P-4QP+Q^2+2Q\right) \eta^3 
	\\
	&+\big(8 P^3 \xi+15 P^2 Q \xi+16 P Q^2 \xi-16 P^2 \xi -15 P Q \xi +8 P \xi\big) \eta^2  \\
		 & +\left(-14 P^3 \xi^2-32 P^2 Q \xi^2-27 P Q^2 \xi^2+32 P^2 \xi^2+45 P Q \xi^2-18 P \xi^2\right) \eta  \\
	 &+6 P^3 \xi^3+15 P^2 Q \xi^3+12 P Q^2 \xi^3-15 P^2 \xi^3-24 P Q \xi^3+12 P \xi^3 \Big). 
	\end{aligned}
\end{equation}
Notice that 
	\begin{equation}\label{constr:1}
		\begin{aligned}
			&P+Q<1, \quad 
			P\geq0,\quad Q\geq0,\\
			&\eta>0,\quad \xi>0, \quad 
			\eta-\xi=B^2>0. 
		\end{aligned}
	\end{equation}
This implies $2QP^2-Q^2P-4QP+Q^2+2Q=2Q(P-1)^2+Q^2(1-P)\geq0$, which along with \eqref{ineq:4} gives 
	\begin{equation}\label{ineq:5}
	\begin{aligned}
		\mathcal W^4(\xi) \phi(\xi) \leq& -\frac{1}{\xi\eta^3}\left[\left(8 P^3 \xi+15 P^2 Q \xi+16 P Q^2 \xi-16 P^2 \xi-15 P Q \xi+8 P \xi\right)\eta^2\right. \\
		& +\left(-14 P^3 \xi^2-32 P^2 Q \xi^2-27 P Q^2 \xi^2+32 P^2 \xi^2+45 P Q \xi^2-18 P \xi^2\right) \eta\\
		& \left.+6 P^3 \xi^3+15 P^2 Q \xi^3+12 P Q^2 \xi^3-15 P^2 \xi^3-24 P Q \xi^3+12 P \xi^3\right]\\
		= &-\frac{1}{\eta^3}P \widehat \phi (P,Q,\xi,\eta),
		\end{aligned}
	\end{equation}
	where
	\begin{equation*}
		\begin{aligned}
		 \widehat \phi(P,Q,\xi,\eta)  := &  8 P^2 \eta^2-14 P^2 \xi \eta+6 P^2 \xi^2+15 P Q \eta^2-32 P Q  \xi\eta+15 P Q \xi^2  \\
		& +16 Q^2 \eta^2-27 Q^2 \xi \eta
			   +12  Q^2 \xi^2-16 P \eta^2+32 P \xi \eta-15 P \xi^2
			     \\
			   &-15 Q \eta^2 +45 Q  \xi \eta-24 Q \xi^2+8 \eta^2-18 \xi \eta+12 \xi^2.
		\end{aligned}
	\end{equation*}
	Therefore, in order to prove \eqref{WKL2}, it is sufficient to prove 
\begin{equation}\label{WKL3}
		\widehat \phi(P,Q,\xi,\eta)\geq0 
\end{equation}
	for all  $P$, $Q$, $\eta$, and $\xi$ satisfying \eqref{constr:1}. 

	As a crucial observation, we notice that $\widehat \phi(P,Q,\xi,\eta)$ is a quadratic form in the variables $\xi$ and $\eta$: 
	$$\widehat \phi (P,Q,\xi,\eta)=\bf z^\top\bf A\bf z,$$
	where
	${\bf z}=(\xi,\eta)^\top$ and 
	\begin{equation*}
		{\bf A}=
		\begin{pmatrix}
			6P^2+12Q^2+15P(Q-1)-24Q+12&-7P^2-\frac{27}2 Q^2-16P(Q-1)+\frac{45}2 \\ 
			-7P^2-\frac{27}2 Q^2-16P(Q-1)+\frac{45}2 Q-9&8P^2+16Q^2+15Q(P-1)-16P+8
		\end{pmatrix}.
	\end{equation*}
	Now, the task of proving \eqref{WKL3} boils down to showing 
 that $\bf A$ is positive semi-definite. Notice that 
	$$6P^2+12Q^2+15P(Q-1)-24Q+12=
	\binom{P}{Q-1}^\top
	\begin{pmatrix}
		6&7.5 \\ 
		7.5&12
	\end{pmatrix}	
	\binom{P}{Q-1}\geq0.$$
	Hence, we only need to prove
	\begin{equation}
		\begin{aligned}
			\varphi (P,Q):=\det({\bf A})=&15 + 2P^2 - \frac{405}{4}Q^2 + 33Q - 24P - 99PQ - 14P^3Q - 28P^2Q^2   \\
			&- 12PQ^3+ 80P^2Q + 135PQ^2 - P^4 + \frac{39}{4}Q^4 + 8P^3 + \frac{87}{2}Q^3\geq0
		\end{aligned}
	\end{equation}
	for all $(P,Q)$ in the domain 
	$$ \Theta :=\left\{ (P,Q):   0\leq P\leq1,0\leq Q\leq1\right\},$$ 
	which is larger than the set formed by the constraints \eqref{constr:1}. 
	By solving the equations
	\begin{equation}
			\frac{\partial  \varphi(P,Q)}{\partial P}=0,\qquad 
			\frac{\partial  \varphi(P,Q)}{\partial Q}=0,
	\end{equation}
	we know that the function $\varphi(P,Q)$ has only two stationary points, $(0,1)$ and $(1,0)$, in the domain $ \Theta$. 
	Both these two stationary points are located on the boundary of $\Theta$, implying that the minimum of 
	$\varphi (P,Q)$ on $\Theta$ is attained on the boundary. 
   Notice that 
	\begin{align*}
		\varphi(0,Q)&=15 - \frac{405}{4}Q^2 + 33Q + \frac{39}{4}Q^4 + \frac{87}{2}Q^3=(Q-1)^2\left(\frac{39}{4}Q^2+63Q+15\right)\geq0,\\
		\varphi(1,Q)&=\frac{23}{4}Q^2 + \frac{63}{2}Q^3 + \frac{39}{4}Q^4=Q^2\left(\frac{23}{4} + \frac{63}{2}Q + \frac{39}{4}Q^2\right)\geq0,\\
		\varphi(P,0)&=-P^4 + 8P^3 + 2P^2 - 24P + 15=-(P-1)^2(P^2-6P-15)\geq0,\\
		\varphi(P,1)&=-P^4 - 6P^3 + 54P^2=-P^2(P^2+6P-54)\geq0,
	\end{align*}
which implies 
$$0 \le 
\min_{(P,Q) \in \partial \Theta} \varphi (P,Q) = \min_{(P,Q) \in \Theta}  \varphi (P,Q).$$ 
Therefore, $ \varphi (P,Q) \ge 0$ for all $(P,Q) \in \Theta$. 
Hence $\bf A$ is positive semi-definite, and inequality \eqref{WKL3} is true. 
Thanks to \eqref{ineq:5}, we have $\mathcal W^4(\xi) \phi(\xi) \le 0$, which yields \eqref{WKL2}. 
In conclusion, inequality \eqref{eq:ineq_g}, or equivalently, \eqref{eq:ineq} is true. The proof is completed. 
\end{proof}

\subsubsection{Complete proofs of Theorems \ref{thm:initialset} and \ref{thm:initialset_gEOS}}\label{sec:cproof}

\begin{proof}
	For the $\gamma$-law EOS, 
the inequality \eqref{eq:ineq} established in Theorem \ref{lem:inequality} implies 
	\begin{equation*}
		\frac{d}{d \xi}\Big(\xi^4{\mathcal F}''(\xi)\Big)=\xi^3 \Big( \xi	{\mathcal F}'''(\xi)+4{\mathcal F}''(\xi) \Big)>0, \qquad \forall \xi \in \Omega_2.
	\end{equation*}
	It means that $\xi^4{\mathcal F}''(\xi)$ is strictly increasing on the interval $\Omega_2=(\xi_b, +\infty)$. This is a critical observation. 
	Hence, it must satisfy one of the following two conditions:
	\begin{itemize}[leftmargin=*]
		\item $\xi^4{\mathcal F}''(\xi)$ does not change sign in $\Omega_2$, or equivalently, ${\mathcal F}'(\xi)$ is monotone.
		\item There exists $\xi_{in}\in\Omega_2$, such that 
		$\xi_{in}^4{\mathcal F}''(\xi_{in})=0$, $\xi^4{\mathcal F}''(\xi) < 0$ for all $\xi \in (\xi_b,\xi_{in})$, and $\xi^4{\mathcal F}''(\xi) > 0$ for all $\xi \in (\xi_{in},+\infty)$. This means
		$${\mathcal F}''(\xi_{in})=0, \quad  {\mathcal F}''(\xi) < 0 ~\mbox{for~all}~ \xi \in (\xi_b,\xi_{in}), \quad  {\mathcal F}''(\xi) > 0 ~\mbox{for~all}~ \xi \in (\xi_{in},+\infty). 
$$		
	\end{itemize}  
Therefore, for the $\gamma$-law EOS, ${\mathcal F}'(\xi)$ is monotone on $\Omega_2 = (\xi_b, +\infty)$, or there is an inflection point $\xi_{in} \in \Omega_2$ such that ${\mathcal F}'(\xi)$ is monotonically decreasing on $ (\xi_b,\xi_{in}]$ and increasing on $[\xi_{in},+\infty)$. This is also the assumption in Theorem \ref{thm:initialset_gEOS} for a general EOS. 
Recall that 
Theorem \ref{thm:m3As} indicates ${\mathcal F}(\xi_b)<0$, $\lim\limits_{\xi\to +\infty}{\mathcal F}(\xi)>0$, and ${\mathcal F}'(\xi)>0$ for all $\xi \in (\xi_b, +\infty)$. 
Hence, all the requirements in Lemma \ref{lem:NRconverges} are satisfied by ${\mathcal F}(\xi)$. 
Thanks to Lemma \ref{lem:NRconverges}, we know that 
if the initial value $\xi_0\in \Omega_3=(\xi_b,\xi_{*}]$, then the iteration sequence generated by the NR method \eqref{eq:NR} satisfies 
$$\left\{\xi_n\right\}_{n\geq0}\subseteq \Omega_2=(\xi_b, +\infty), \qquad \lim\limits_{n\to +\infty}\xi_n=\xi_{*},$$ 
which indicate the PCP property (by Theorem \ref{thm:PCPset}) and convergence. 
The proofs of Theorems \ref{thm:initialset} and \ref{thm:initialset_gEOS} are completed. %
\end{proof}

\subsection{Theories on finding an initial value $\xi_0$ in the safe interval $\Omega_3$}\label{sec:initialvalue}
Theorems \ref{thm:initialset} and \ref{thm:initialset_gEOS} tell us that if the initial guess $\xi_0\in\Omega_3=(\xi_b,\xi_*]$, then the 
NR method \eqref{eq:NR} is PCP and convergent. 
However, the right endpoint $\xi_*$ of this ``safe'' interval $\Omega_3$ is the (unknown) exact root 
 of ${\mathcal F}(\xi)$, while the left endpoint $\xi_b$ is a root of the quartic polynomial $f_b(\xi)$, which is difficult to calculate efficiently. 
 Hence, it is nontrivial to determine a computable initial value $\xi_0$ within $\Omega_3$. 

Interestingly, we discover that the cubic polynomial $f_c(\xi)$ in \eqref{eq:f3} 
has a unique positive root (denoted by $\xi_c$), which is always located within the ``safe'' interval $\Omega_3$.

\begin{lemma}[see Lemma 2.6 in \cite{wu2017admissible}]\label{lem:f3oneroot}
	Consider a general EOS \eqref{express:enthalpy} satisfying \eqref{eq:EOScond}. Assume that ${\bf U} \in {\mathcal G}$. 
	The cubic polynomial $f_c(\xi)$ in \eqref{eq:f3}  has a unique positive root $\xi_c$. Moreover, $f_c(\xi)<0$ for all $\xi \in(0,\xi_c)$, and $f_c(\xi)>0$ for all $\xi \in (\xi_c,+\infty)$.
\end{lemma}

\begin{theorem}\label{thm:f3omega3}
	Consider a general EOS \eqref{express:enthalpy} satisfying \eqref{eq:EOScond}. 
	If ${\bf U} \in {\mathcal G}$, then 
 the positive root $\xi_c$ of $f_c(\xi)$ always lies in 
	the ``safe'' interval 
	$\Omega_3=(\xi_b,\xi_*]$. 
\end{theorem}

\begin{proof} 
	According to Theorem \ref{thm:m3As}, we have 
	$$f_b(\xi_*)=\mathcal W^{-2}(\xi_*^2-D^2\mathcal W^2(\xi_*))(\xi_*+B^2)^2>0,$$
	which yields $\xi_*>D\mathcal W(\xi_*)$. 
	We observe that 
	\begin{align*}
		0= \xi_*^2{\mathcal F}(\xi_*)=&\xi_*^2\left(\xi_*-\mathcal{P}\left(\frac{D}{\mathcal W(\xi_*)},\frac{\xi_*}{D\mathcal W(\xi_*)}\right)+B^2-\frac{1}{2}\left(\frac{B^2}{\mathcal W^2(\xi_*)}+\frac{\tau^2}{\xi_*^2}\right)-E\right)\\
		=&\xi_*^3+(B^2-E)\xi_*^2-\frac{1}{2}\left(\frac{B^2\xi_*^2}{\mathcal W^2(\xi_*)}+\tau^2\right)-p \xi_*^2\\
		=&f_c(\xi_*)-\frac{B^2}{2\mathcal W^2(\xi_*)}(\xi_*^2-D^2\mathcal W^2(\xi_*))-p\xi_*^2,
	\end{align*}
	where $p>0$ denotes the exact pressure corresponding to ${\bf U}$. 
It follows that 
	$$f_c(\xi_*)=\frac{B^2}{2\mathcal W^2(\xi_*)}(\xi_*^2-D^2\mathcal W^2(\xi_*))+p \xi_*^2> \frac{B^2}{2\mathcal W^2(\xi_*)}(\xi_*^2-D^2\mathcal W^2(\xi_*)) > 0.$$
	By Lemma \ref{lem:f3oneroot}, we obtain $\xi_* > \xi_c$. 
	Since $\mathbf U\in\mathcal{G}$, from the inequality (2.17) in \cite{wu2017admissible}, we have $\xi_b<\xi_c$ (which corresponds to $\xi_4(\mathbf U)<\xi_3(\mathbf U)$ for the notations in \cite{wu2017admissible}). 
	Combining $\xi_*>\xi_c$ with $\xi_b<\xi_c$, we obtain $\xi_c\in\Omega_3=(\xi_b,\xi_*].$ The proof is completed. 
\end{proof}


According to a variant \cite{fan1989cubic} of Cardano's formula for cubic equations, we obtain a real analytical expression for $\xi_c$, as summarized in \eqref{eq:xi_c}.

\begin{theorem}\label{thm:expressxic} 
	Consider a general EOS \eqref{express:enthalpy} satisfying \eqref{eq:EOScond}. 
	Define $\delta=27a_0+4 \alpha_1^3$ with $\alpha_1=B^2-E$ and $a_0=-0.5(B^2D^2+\tau^2)$.  
	\begin{itemize}[leftmargin=*]
		\item If $\delta>0$, then $\xi_c=-\frac{\alpha_1}{3}\left(1-2\cos\left(\frac{\theta}{3}-\frac{\pi}{3}\right)\right)$, where $\theta=\arccos\left(1+\frac{13.5a_0}{\alpha_1^3}\right)$.
		\item If $\delta\le 0$, then $\xi_c=-\frac13{\left(\alpha_1+\sqrt[3]{X_1+X_2}+\sqrt[3]{X_1-X_2}\right)}$, where $X_1=\alpha_1^3+13.5a_0$ and $X_2=1.5\sqrt{3a_0\delta}$.
	\end{itemize}
\end{theorem}

\begin{proof}
	 We consider the following three cases.
	\begin{itemize}[leftmargin=*]
		\item If $\delta=27a_0+4 \alpha_1^3=-13.5(B^2D^2+\tau^2)+4 \alpha_1^3>0$, then $\alpha_1>0$, and 
		$f_c(\xi)$ has three real roots $-\frac{\alpha_1}{3}\left(1-2\cos\left(\frac{\theta}{3}\pm\frac{\pi}{3}\right)\right)$ and $-\frac{\alpha_1}{3}\left(1+2\cos\left(\frac{\theta}{3}\right)\right)$, where $\theta=\arccos\left(1+\frac{13.5a_0}{\alpha_1^3}\right)$. 
		Since $\theta\in[0,\pi]$, we have $1-2\cos\left(\frac{\theta}{3}-\frac{\pi}{3}\right)<0$. Thus   $-\frac{\alpha_1}{3}\left(1-2\cos\left(\frac{\theta}{3}-\frac{\pi}{3}\right)\right)>0$. Recall that $f_c(\xi)$ has only one positive root $\xi_c$. Therefore,   $\xi_c=-\frac{\alpha_1}{3}\left(1-2\cos\left(\frac{\theta}{3}-\frac{\pi}{3}\right)\right)$. 
		\item If $\delta<0$, then $f_c(\xi)$ has two non-real complex conjugate roots and one real root $\xi_c$, which can be expressed as  %
		$-\frac13{\left(\alpha_1+\sqrt[3]{X_1+X_2}+\sqrt[3]{X_1-X_2}\right)}$, according to \cite{fan1989cubic}.
		\item  If $\delta=0$, then $f_c(\xi)$ has three real roots, with two being repeated roots $\frac{9a_0}{2\alpha_1^2}$ and one being a single root $-\alpha_1-\frac{9a_0}{\alpha_1^2}$. Note that $a_0=-0.5(B^2D^2+\tau^2)<0$, implying that $\frac{9a_0}{2\alpha_1^2}<0$. Hence, the unique positive root $\xi_c$ must be $-\alpha_1-\frac{9a_0}{\alpha_1^2}$, which can also be expressed as $-\frac13{\left(\alpha_1+\sqrt[3]{X_1+X_2}+\sqrt[3]{X_1-X_2}\right)}$.
	\end{itemize}
\end{proof}

Theorem \ref{thm:f3omega3} indicates that 
the unique positive root $\xi_c$ of $f_c(\xi)$ always lies in 
the ``safe'' interval 
$\Omega_3=(\xi_b,\xi_*]$. According to Theorem \ref{thm:initialset}, 
if we take 
 the initial guess $\xi_0= \xi_c$, then the 
NR method \eqref{eq:NR} is always PCP and convergent. 
Although $\xi_c$ can be computed by using the analytical formula \eqref{eq:xi_c} (proven in Theorem \ref{thm:expressxic}), it involves expensive trigonometric operations and cubic root calculations. 
To reduce the computational cost, we find a cheaper initial value $\xi_d$ defined in \eqref{eq:xi_d}, which satisfies the following relation \eqref{eq:4xi}.

\begin{theorem}\label{lem:xid}
	For a general EOS \eqref{express:enthalpy} satisfying \eqref{eq:EOScond} and any ${\bf U} \in {\mathcal G}$, it holds 
\begin{equation}\label{eq:4xi}
		\xi_d>\xi_c>\xi_b>\xi_a.
\end{equation}
\end{theorem}
\begin{proof}
	It was shown in \cite[Eq.~(2.21)]{wu2017admissible} that $\xi_c<\xi_d$ (which corresponds to $\xi_3(\mathbf U)<\xi_{2,R}(\mathbf U)$ for the notations in \cite{wu2017admissible}). 
	Recall that $\xi_b>\xi_a$ (proven in Theorem \ref{thm:m3As}) and $\xi_c>\xi_b$ (proven in Theorem \ref{thm:f3omega3}). 
	Combining these results completes the proof. 
\end{proof}

Although $\xi_d$ is computationally cheaper than $\xi_c$ and satisfies $\xi_d> \xi_b$, 
there are a few instances where $\xi_d$ may be larger than $\xi_*$ and thus fall outside the ``safe'' interval $\Omega_3 = (\xi_b, \xi_*]$. In our random tests, this scenario occurs approximately in 19\% of cases. 
To optimize efficiency while maintaining the PCP property and convergence, we propose the following hybrid strategy for selecting the initial estimate $\xi_0$:
\begin{enumerate}
	\item If ${\mathcal F}(\xi_d)\le 0$, then $\xi_d \le \xi_*$ due to the monotonicity of ${\mathcal F}(\xi)$ in Theorem \ref{thm:m3As}. In this case, $\xi_d \in \Omega_3=(\xi_b,\xi_*]$, and we take the initial guess $\xi_0=\xi_d$. 
	\item If ${\mathcal F}(\xi_d) > 0$, then we choose the initial guess  $\xi_0=\xi_c$.
\end{enumerate}
This hybrid strategy leads to our robust and efficient initial guess defined in \eqref{eq:initial}, ensuring the efficiency, PCP property, and convergence.


\subsection{Insights for developing other PCP solvers for primitive variables}
	While the focus of this paper is on the NR method, our findings extend beyond this specific approach and are broadly applicable to the development of other root-finding algorithms for computing the root $\xi_*$ of ${\mathcal F}(\xi)$. For instance, we have found $\xi_c$ as a computable lower bound for $\xi_*$. Moreover, as indicated by Proposition \ref{prop:upper}, $2 E - B^2$ serves as a simple upper bound for $\xi_*$. The bounded interval $(\xi_c, 2 E-B^2)$ lies within $\Omega_2$, ensuring adherence to the physical constraints \eqref{eq:Qphysical}, as proven in Theorem \ref{thm:PCPset}. These findings pave the way for the design of other PCP convergent methods, such as the bisection or Brent's algorithms, for robust recovery of primitive variables in RMHD.

\begin{Prop}\label{prop:upper}
	For a general EOS \eqref{express:enthalpy} satisfying \eqref{eq:EOScond} and any ${\bf U} \in {\mathcal G}$, we have $\xi_* < 2E-|{\bf B}|^2$.
\end{Prop}

\begin{proof}
	Recalling that $\xi_* = \rho h W^2$, we have 
\begin{align*}
		(2E - |{\bf B}|^2) - \xi_* & = 2 \left( \rho h W^2 - p + \frac{1}{2} |{\bf v}|^2
		|{\bf B}|^2 - \frac12 ({\bf v} \cdot {\bf B})^2  \right)  - \xi_*  \\
		& \ge 2 ( \rho h W^2 - p ) - \xi_* 
		= \rho h W^2 - 2 p \ge \rho {\mathcal H}(\rho,p) - 2p > 0,
\end{align*}
	where the second condition in \eqref{eq:EOScond} has been used in the last step. 
\end{proof}


\subsection{Broad applicability of PCP NR method}
	The proposed PCP NR method demonstrates notable versatility and can be seamlessly integrated into any conservative numerical scheme for RMHD. As an application, we will combine it with the provably PCP discontinuous Galerkin (DG) schemes proposed in \cite{wu2017admissible,wu2021provably}. These DG schemes, incorporating a PCP limiter, were rigorously proven to maintain the admissibility of the computed conservative variables, ${\bf U} \in {\mathcal G}$, for both cell averages and nodal point values. This pivotal attribute ensures the existence and uniqueness of the physical root $\xi_*$ of ${\mathcal F}(\xi)$, thereby theoretically guaranteeing the existence and uniqueness of the corresponding physical primitive vector. Thanks to the inherent PCP property and its convergence, our robust NR method ensures that the recovered primitive variables are physically admissible. 
	This integration leads to the fully PCP schemes, ensuring all the computational processes in RMHD adhere to physical constraints \eqref{eq:Qphysical}. 

\section{Numerical Tests}\label{sec:numerical_tests}
This section presents several numerical experiments to validate the accuracy, efficiency, and PCP property of the PCP NR method. We compare this method's performance with other primitive variable solvers from the literature. For simplicity, we use the following abbreviations: ``PCP NR" for our NR method, ``NH-FP" and ``NH-FP-Aitken" for the fixed-point iteration algorithms without/with Aitken acceleration from \cite{newman2014primitive}, ``PL-Brent" for Brent's method from \cite{palenzuela2015effects}, ``NG-2DNR" for the 2D NR method from \cite{noble2006primitive}, ``MM-1DNR" for the 1D NR method from \cite{mignone2007equation}, ``GR-2DNR" for the 2D NR method from \cite{giacomazzo2007whiskymhd}, and ``CF-2DNR" for the 2D NR method from \cite{cerda2008new}. We set the accuracy threshold $\epsilon_{\tt target}=10^{-14}$ and cap the maximum iterations at $500$ for each algorithm. All tests are implemented in C++ with double precision and conducted on a Linux server with an Intel(R) Xeon(R) Gold 5318Y CPU at 2.10 GHz.

To assess the above algorithms' reliability and robustness, we define specific criteria. An algorithm is considered to {\em fail} in a test if it diverges, produces non-real complex numbers, or the final result is a nonphysical primitive vector violating constraints \eqref{eq:Qphysical}. Additionally, we categorize the iterative process as {\em non-PCP} if it yields quantities corresponding to a nonphysical primitive vector at any iteration stage.

\subsection{Random tests for $\gamma$-law EOS}\label{sec:randomtestgamma}

To validate the accuracy, efficiency, and robustness of our PCP NR method for a wide range of cases, we conduct two sets of random tests with $\gamma=1+U_{\tt rand}$. Here and hereafter  
 $U_{\tt rand}$ denotes the uniform random
variables independently generated in $[0,1]$; {\em the values of $U_{\tt rand}$ vary at different locations.} 
In our experiments, we first randomly generate primitive variables $\mathbf Q = (\rho, \mathbf v^\top, \mathbf B^\top, p)$ and then calculate the corresponding conservative variables $\mathbf{U}$. 
Subsequently, we apply the primitive variable solvers to recompute/recover the approximate primitive variables $\mathbf Q' = (\rho', \mathbf v'^\top, \mathbf B^\top, p')$ from $\mathbf{U}$. Numerical errors are then assessed using $|\mathbf v' - \mathbf v|$. 
In the first set of random tests, we generate primitive variables as follows:
\begin{equation}\label{test:rand1}
	\begin{cases}
		\rho=1000 U_{\tt rand}+10^{-11},\\
		\mathbf v=(1-10^{-10})U_{\tt rand}\mathbf u/|\mathbf u| \mbox{  with  } \mathbf u=2(U_{\tt rand}, U_{\tt rand}, U_{\tt rand})^\top-(1,1,1)^\top,\\
		p=1000 U_{\tt rand}+10^{-11},\\
		\mathbf B=200(U_{\tt rand}, U_{\tt rand}, U_{\tt rand})^\top-100(1,1,1)^\top.
	\end{cases}
\end{equation}
The second set of random tests involves low density/pressure and high Lorentz factors:
\begin{equation}\label{test:rand3}
\begin{cases}
	\rho=0.01 U_{\tt rand}+10^{-13},\\
	\mathbf v=((0.01-10^{-16})U_{\tt rand}+0.99)\mathbf u/|\mathbf u|  \mbox{  with  }  \mathbf u=2(U_{\tt rand}, U_{\tt rand}, U_{\tt rand})^\top-(1,1,1)^\top,\\
	p=0.01 U_{\tt rand}+10^{-13},\\
	\mathbf B=20(U_{\tt rand}, U_{\tt rand}, U_{\tt rand})^\top-10(1,1,1)^\top.
\end{cases}
\end{equation}

\begin{table}[!t]
	\centering
	\captionsetup{font=small}
	\caption{The first set of random tests: non-PCP iteration processes, algorithm failure counts, average/maximum 
		iteration counts,  average/maximum errors, and average CPU time ($\times 10^{-7}$s) of successful iterations, in $10^8$ independent random experiments. 
	}\label{table:recovering algorithms tests1:100}
	
	\begingroup
\setlength{\tabcolsep}{2.6666pt} 
\renewcommand{\arraystretch}{1.2} 

\centering

\small
\begin{tabular}{ccccccccc}
	\bottomrule[1.0pt]
	algorithm & non-PCP \#  & failure \#  & ave ite \# & max ite \# & ave error  & max error  & CPU time \\ \hline
		PCP NR  &0       & 0    & 4.8 & 15 & 3.1e-16       & 2.5e-12    & 4.69                             \\
		NH-FP-Aitken& 0         & 0 & 9.0   & 500 & 4.8e-16         & 1.7e-12 & 9.23                            \\
		NH-FP& 0          & 0    & 34.4     & 500  & 2.7e-13          & 1.1e-05    & 35.25                       \\
		PL-Brent& 3.8e7         & 0    & 9.7   & 30   & 3.4e-16         & 1.9e-12    & 13.45                          \\
		MM-1DNR& 1.5e3         & 0    & 6.5   & 268   & 2.9e-16         & 1.9e-12    & 5.12                     \\ 
		NG-2DNR& 4.9e5         & 1.0e5    & 4.6      & 66    & 2.8e-16         & 2.1e-12    & 8.61                    \\ 
		GR-2DNR& 8.2e6        & 5.0e5    & 9.7     & 500   & 2.9e-16         & 2.8e-12    & 5.71                     \\ 
		CF-2DNR& 7.2e6         & 2.1e4    & 5.1     & 500   & 2.9e-16         & 4.0e-12    & 3.49                     \\ 
		
		\toprule[1.0pt]
	\end{tabular}
	
	\endgroup
	
\end{table}

\begin{table}[!t]
	\centering
	\captionsetup{font=small}
	\caption{Same as Table \ref{table:recovering algorithms tests1:100} except for the second set of random tests.}\label{table:recovering algorithms tests3:1}
	
	\begingroup
\setlength{\tabcolsep}{2.6666pt} 
\renewcommand{\arraystretch}{1.2} 

\centering

\small
		\begin{tabular}{ccccccccc}
	\bottomrule[1.0pt]
	algorithm & non-PCP \#  & failure \#  & ave ite \# & max ite \# & ave error  & max error  & CPU time \\ \hline
	PCP NR  &0       & 0    & 5.0 & 17 & 3.6e-14       & 3.3e-10    & 4.65                             \\
	NH-FP-Aitken& 1         & 1 & 12.4   & 500 &1.1e-13         & 1.7e-06 & 12.36                            \\
	NH-FP& 0          & 0    & 55.0     & 500  & 2.0e-09          & 9.0e-04    & 57.69                       \\
	PL-Brent& 9.8e7         & 0    & 25.8   & 51   & 4.1e-14         & 6.6e-10    & 30.13                          \\
	MM-1DNR& 6.4e2         & 0    & 6.9   & 454 & 4.1e-14         & 5.3e-10    & 5.26                      \\ 
	NG-2DNR& 5.9e7         & 5.6e7    & 15.7   & 500  & 4.2e-14         & 2.6e-10    & 21.71                    \\ 
	GR-2DNR& 5.6e7        & 5.5e7    & 234.3     & 500   & 4.3e-14         & 3.2e-10    & 83.28                     \\ 
	CF-2DNR& 2.1e7         & 2.0e7    & 203.0     & 500   & 5.3e-14         & 3.5e-07    & 65.67                     \\ 
	
	\toprule[1.0pt]
\end{tabular}
	
	\endgroup
	
\end{table}

Our PCP NR method uses the robust initial guess \eqref{eq:initial}. 
The selection of reliable initial guesses for the 2D NR methods (CF-2DNR, GR-2DNR, and NG-2DNR) 
is important but currently lacks both theoretical and empirical guidance. A prevalent strategy in the literature is to take the initial value from the previous time step in numerical evolution  \cite{gammie2003harm,cerda2008new,noble2006primitive,siegel2018recovery}.  
However, the effectiveness of this strategy is unpredictable and challenging to evaluate \cite{kastaun2021robust} in our random tests, as the initial guess is not chosen in a deterministic way.  
Therefore, in our random tests, we introduce a (small) 10\% random perturbation to the exact primitive variables 
to serve as the initial guess for these 2D NR methods. 
For the CF-2DNR method, if the random initial guess causes failure, we then restart the NR iterations  
with another initial guess proposed in \cite{cerda2008new}.

The results of our tests, which consist of $10^8$ independent random
experiments, are presented in Tables \ref{table:recovering algorithms tests1:100} and \ref{table:recovering algorithms tests3:1}. 
We  
enumerate the total counts of non-PCP iteration processes and algorithmic failures. Additionally, these tables include average and maximum iteration counts, average and maximum numerical errors, and the average CPU time for successful iterations of the $10^8$ experiments. 
We observe that the PCP NR and NH-FP methods consistently maintain PCP throughout the iterations and exhibit no failures in all tests, underscoring their exceptional robustness. NH-FP-Aitken is generally robust but fails in only one test. 
The PL-Brent and MM-1DNR methods, although not always PCP, do not fail in these test cases, successfully yielding physical primitive variables. In contrast, the 2D NR methods (CF-2DNR, GR-2DNR, and NG-2DNR) 
demonstrate a lack of robustness, frequently failed to maintain PCP, and encountered numerous failures.
Notably, the PCP NR method consistently delivers accurate results and emerges as the most efficient among all tested algorithms. It requires the fewest iteration steps and generally the least CPU time. 
The experimental results indicate that the PCP NR method outperforms the others in terms of robustness, efficiency, and accuracy.

Furthermore, our results indicate that in about 81\% of the random test cases, the condition ${\mathcal F}(\xi_d) \le 0$ in \eqref{eq:initial} is satisfied in our PCP NR method. This implies that the complex computation of $\xi_c$ is unnecessary in these instances, thereby enhancing the efficiency of our method.



\subsection{Random tests for Mathews-EOS and RC-EOS}

Theorem \ref{thm:initialset_gEOS} demonstrates that the PCP property and the convergence of the proposed PCP NR method are provable under a specific assumption on ${\mathcal F}'(\xi)$. To validate this assumption, we conduct $10^8$ random experiments described in \eqref{test:rand1}–\eqref{test:rand3}, for the Mathews-EOS (\cite{mathews1971hydromagnetic})
${\mathcal H}(\rho,p)=\frac{5p}{2\rho}+\sqrt{\frac{9p^2}{4\rho^2}+1}$ 
and RC-EOS (\cite{Ryu_2006})
${\mathcal H}(\rho,p)=\frac{2(6p^2+4p\rho+\rho^2)}{\rho(3p+2\rho)}.$ 
Our extensive random tests have yielded two distinct structural patterns for ${\mathcal F}'(\xi)$. In the first scenario, ${\mathcal F}'(\xi)$ exhibits monotonic increase within the interval $\Omega_2 = (\xi_b, +\infty)$. In the second case, we observe an inflection point $\xi_{in} \in \Omega_2$, where ${\mathcal F}'(\xi)$ decreases monotonically over the interval $(\xi_b,\xi_{in}]$ and increases over $[\xi_{in},+\infty)$. Figure \ref{fig:Df} provides a graphical illustration of these findings. These results lend credence to the assumption regarding ${\mathcal F}'(\xi)$ in Theorem \ref{thm:initialset_gEOS}. However, a rigorous proof of this assumption for a general EOS remains elusive. 
To validate the effectiveness of the PCP NR method for the Mathews-EOS and RC-EOS, Table \ref{table:recovering algorithms tests:ryu's eos} presents the results of $10^8$ independent random tests across two distinct sets \eqref{test:rand1}–\eqref{test:rand3}. The results clearly show that our PCP NR algorithm consistently maintains the PCP property and successfully recovers the physical primitive variables, demonstrating its robustness for both the Mathews-EOS and RC-EOS.

\begin{table}[!t]
	\centering
	\captionsetup{font=small}
	\caption{Two sets of tests for Mathews-EOS and RC-EOS: non-PCP iteration, algorithm failure counts, average iteration counts, and average errors in $10^8$ independent random experiments.}\label{table:recovering algorithms tests:ryu's eos}
	
	\begingroup
	\setlength{\tabcolsep}{6pt} 
	\renewcommand{\arraystretch}{1} 
	
	\centering
	
	\small
	\begin{tabular}{cccccc}
	\bottomrule[1.0pt]		
	&EOS & non-PCP ite \# & failure \# & ave ite \# & ave error \\ \hline
	\multirow{2}{*}{Test 1}	&RC-EOS & 0 & 0 &4.9 & 2.7e-16 \\  
	
	&Mathews-EOS & 0 & 0	& 4.9  & 2.8e-16 \\ \hline
	
	\multirow{2}{*}{Test 2}	&RC-EOS & 0 & 0 &4.1 & 2.4e-14 \\  
	
	&Mathews-EOS & 0 & 0	& 4.1  & 2.5e-14 \\
	
	\toprule[1.0pt]
\end{tabular}
	
	\endgroup
	
\end{table}

\begin{figure}[!t]
	\centering
	\subfloat[$\gamma$-law EOS \eqref{eq:gamma-EOS}]{	\includegraphics[width=0.31\textwidth]{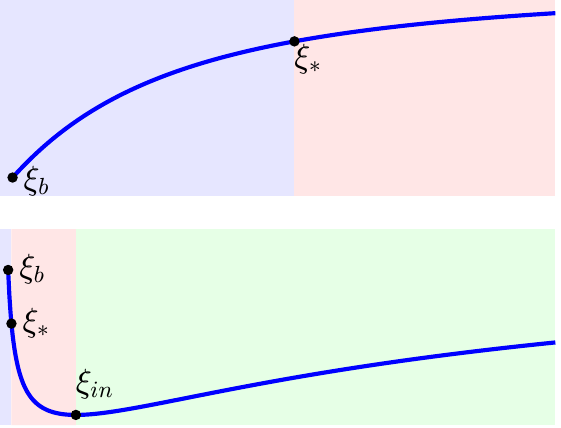} }~
	\subfloat[RC-EOS \cite{Ryu_2006}]{	\includegraphics[width=0.31\textwidth]{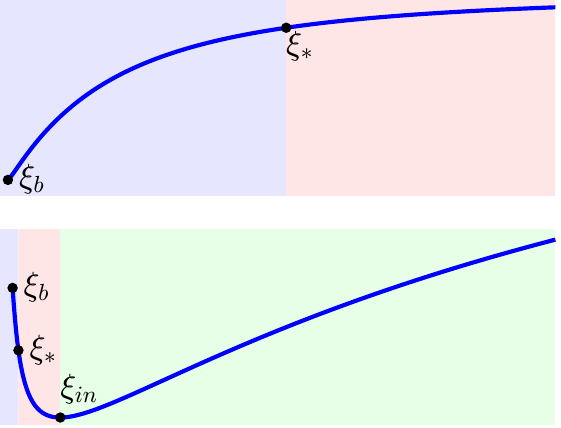} }~
	\subfloat[Mathews-EOS \cite{mathews1971hydromagnetic}]{	\includegraphics[width=0.31\textwidth]{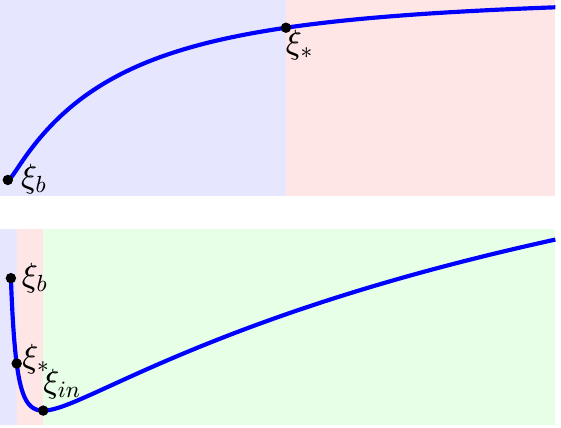} }~
	\caption{Two patterns of ${\mathcal F}'(\xi)$ observed in extensive random experiments for three EOSs.}\label{fig:Df}
\end{figure}

\subsection{Efficiency tests in simulating RMHD problems}

The PCP NR method is highly adaptable and can be seamlessly incorporated into any RMHD scheme involving primitive variable recovery. 
As an application, we have successfully implemented it into the provably PCP DG schemes   \cite{wu2017admissible,wu2021provably}, which were rigorously proven to preserve the computed  ${\bf U} \in {\mathcal G}$. By incorporating the PCP NR method into these DG schemes, we ensure that the primitive variables are also physically admissible. In this and the subsequent subsections, we present several demanding examples to validate the efficiency and robustness of the integrated PCP NR-DG schemes for RMHD simulations.

First, we consider two smooth problems. The initial condition for the first (1D) problem is  $\mathbf{Q}(x,0)=\left(1,0,v_2,v_3,1,\kappa v_2,\kappa v_3,10^{-2}\right)^\top$, 
where $v_2=0.99\sin(2\pi x)$, $v_3=0.99\cos(2\pi x)$, and $\kappa=\sqrt{1+\rho hW^2}$. The computational domain $[0,1]$ is divided into 640 uniform cells. 
 The second (2D) problem describes the Alfv{\'e}n waves with  $\mathbf{Q}(x,0)=\left(1,v_1,v_2,v_3,\cos\left(\frac{\pi}{4}\right)+\kappa v_1,\sin\left(\frac{\pi}{4}\right)+\kappa v_2,\kappa v_3,0.1\right)^\top$,  
where $v_1=0.9\sin(2\pi \theta)\sin\left(\frac{\pi}{4}\right)$, $v_2=0.9\sin(2\pi \theta)\cos\left(\frac{\pi}{4}\right)$, $v_3=0.9\cos(2\pi \theta)$, and $\theta=x\cos\left(\frac{\pi}{4}\right)+y\sin\left(\frac{\pi}{4}\right)$. The computational domain  $[0,\sqrt{2}]^2$ is divided into $100\times100$ uniform rectangular cells. 
The stop time $T_{\tt stop}$ 
is set to 1 for the 1D problem and 0.1 for the 2D problem. Periodic boundary conditions and the $\gamma$-law EOS with $\gamma=5/3$ are employed. 
We compare the performance of four robust primitive variable solvers: the PCP NR, NH-FP-Aitken, MM-1DNR, and PL-Brent methods. These are incorporated into the 1D $\mathbb{P}^2$-based PCP DG scheme from \cite{wu2017admissible} and the 2D $\mathbb{P}^2$-based PCP DG scheme from \cite{wu2021provably}. Table \ref{table:smooth pb simulation} shows the primitive variables' recovery time and the total simulation time. It is observed that among these four solvers, the PCP NR method exhibits the highest efficiency.

	\begin{table}[!tbh]
		
		\centering
		\captionsetup{font=small}
		\caption{The CPU time in recovering primitive variables and the whole simulation CPU time in seconds for simulating two smooth RMHD problems.}\label{table:smooth pb simulation}
		
		\begingroup
		\setlength{\tabcolsep}{6pt} 
		\renewcommand{\arraystretch}{1.2} 
		
		\centering
		
		\small
			%
			%
			%
			%
			%
		
		\begin{tabular}{cccccc}
			\bottomrule[1.0pt]	
			& & PCP NR & NH-FP-Aitken & PL-Brent & MM-1DNR  \\ \hline
			
			\multirow{2}*{1D, $T_{\tt stop}=1$}	
			
			&Total time & 132.54s & 165.39s & 217.61s  & 134.33s \\
			
			&Recovery time & 26.46s & 59.32s & 111.57s & 28.13s \\ \hline
			
			\multirow{2}*{2D, $T_{\tt stop}=0.1$}	
			
			&total time & 153.09s & 213.63s & 275.42s & 157.85s  \\
			
			&recovery time & 39.91s & 99.16s & 161.24s & 44.77s
			
			\\
			
			
			\toprule[1.0pt]
		\end{tabular}
		
		\endgroup
		
	\end{table}

\subsection{Robustness tests in simulating RMHD with shocks}\label{sec:DGexamples}

To further demonstrate the robustness of the integrated PCP NR-DG schemes, we test two ultra-relativistic problems involving strong shocks: a blast problem with $\gamma=4/3$ and a jet with $\gamma=5/3$. 

	\subsubsection{2D blast problem}
This numerical example simulates the evolution of challenging two-dimensional (2D) circular blast waves in a strong magnetic field.   
Our setup follows \cite{BALSARA2016357,wu2021provably}, but our magnetic field $B_a$ is much stronger than that in  \cite{BALSARA2016357}. 
It is well-known that higher values of $B_a$ significantly increase the difficulty of simulating the RMHD  blast wave problem. 
In our study, we set $B_a=2000$, which corresponds to an extremely low plasma-beta of $2.5 \times 10^{-10}$. Note that the classic blast problem \cite{BALSARA2016357} is much milder and typically takes $B_a=1$, which corresponds to a plasma-beta of $0.1$. For strongly magnetized blast problems, many existing numerical schemes in the literature require artificial treatments; however, our integrated PCP NR-DG scheme demonstrates robust performance in this challenging test. Figure \ref{fig:BL} presents our numerical results obtained by $\mathbb{P}^2$-based PCP NR-DG scheme, which are in alignment with those reported in \cite{wu2021provably}. 
Table \ref{table:Robustness in practical simulation} compares the performance of seven different primitive variable solvers integrated into the DG schemes described in \cite{wu2021provably}. 
It indicates that the three 2D NR solvers (NG-2DNR, GR-2DNR, and CF-2DNR) and NH-FP-Aitken encounter failures in this test, whereas the other three solvers successfully recover the primitive variables throughout the simulations. The failure of NH-FP-Aitken method is caused by the Aitken acceleration 
at $t\approx 1.098$. In this example, serious NR oscillations result in a slow convergence rate for the MM-1DNR algorithm, while the PCP NR algorithm avoids this issue due to the utilization of \eqref{fu:express2} as discussed in Appendix \ref{sec:detail}. Among the three successfully executed solvers, the proposed PCP NR method is the fastest. 
		
		\begin{figure}[!t]
			\centering
				\includegraphics[width=0.32\textwidth]{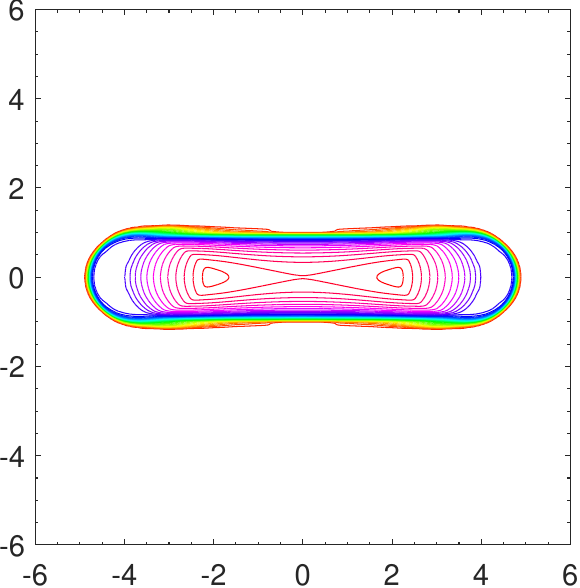} \hfill 
				\includegraphics[width=0.32\textwidth]{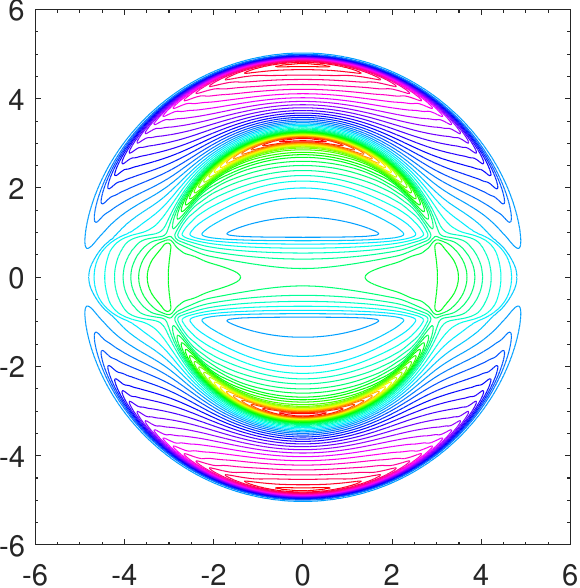} 
				\hfill 
				\includegraphics[width=0.32\textwidth]{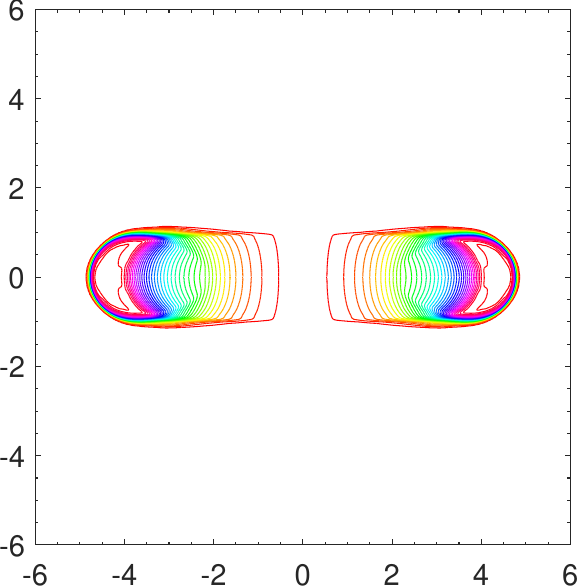}
			\captionsetup{font=small}
			\caption{Contour plots of $\log p$ (left), $|{\bf B}|$ (middle), and $W$ (right) for the blast problem at $t = 4$ obtained by the integrated PCP NR-DG scheme with $400\times400$ rectangular cells.}
			\label{fig:BL}
			\vspace{5mm}
		\end{figure}
	
		\begin{table}[!h]
		
		\centering
		\captionsetup{font=small}

		\caption{Performance of different primitive variable solvers in DG schemes: Execution status, total
			simulation time, and primitive variables recovery time for simulating ultra-relativistic 1D 
			Riemann problems, 2D blast problem, and astrophysical jet. }\label{table:Robustness in practical simulation}
		
		\begingroup
		\setlength{\tabcolsep}{1pt} 
		\renewcommand{\arraystretch}{1.3} 
		
		\centering
		
		\footnotesize
		
		\begin{tabular}{ccccccccc}
			\bottomrule[1.0pt]
			
			&  & PCP NR & NH-FP-Aitken & PL-Brent & MM-1DNR & NG-2DNR & GR-2DNR & CF-2DNR   \\ \hline
%
%
			
			\hline
			
			\multirow{3}{*}{Blast} & Status & Success & Failure & Success & Success & Failure & Failure & Failure   \\
			& Total  & 35h1min & -- & 48h6min & 55h47min & -- & -- & --  \\
			& Recovery  & 2h43min & -- & 16h3min & 23h41min & -- & -- & --  \\
			
			\hline
			
			\multirow{3}{*}{Jet} & Status & Success & Success & Success & Success & Failure & Failure & Failure   \\
			& Total & 106h13min & {112h} & 116h11min & 106h29min & -- & -- & --  \\
			& Recovery  & 12h47min & {19h3min} & 21h47min & 12h48min & -- & -- & --  \\ \toprule[1.0pt]
		\end{tabular}
		
		\endgroup
		
		\vspace{-5 mm}
		
	\end{table}
		
		\subsubsection{Astrophysical jet}
		
In our latest example, we further demonstrate the robustness of our integrated PCP NR-DG scheme through numerical simulations of a challenging relativistic jet originating in astrophysics. 
 While several studies \cite{2015High,WuTang2017ApJS,chen2022physical} have previously simulated relativistic jets without a magnetic field, our focus is on an RMHD jet from \cite{wu2021provably} with $B=\sqrt{2000p}$. Our setup is the same as \cite{wu2021provably} and thus omitted here. 
The computational region is taken as the half domain $[0,10]\times[0,25]$ with $240\times500$ uniform rectangular cells. 
We employ the third-order DG scheme as described in \cite{wu2021provably}, in conjunction with seven distinct solvers for the recovery of primitive variables. 
The comparative performance of these solvers is detailed in Table \ref{table:Robustness in practical simulation}. It reveals that while the three 2D NR solvers experience difficulties and fail during the simulations, the remaining four solvers demonstrate success in accurately recovering the primitive variables throughout the simulation process. Notably, among these four effective solvers, our proposed PCP NR method stands out as the most efficient.
Figure \ref{fig:jet} displays the numerical results at $t=30$ computed using our integrated PCP NR-DG scheme. The results align closely with those in \cite{wu2021provably}, underscoring the efficacy of our approach in simulating this ultra-relativistic, strongly magnetized problem.

		\begin{figure}[!t]
			\centering
			\vspace{-5mm}
			\subfloat[$\log p$]{
				\includegraphics[width=0.3\textwidth]{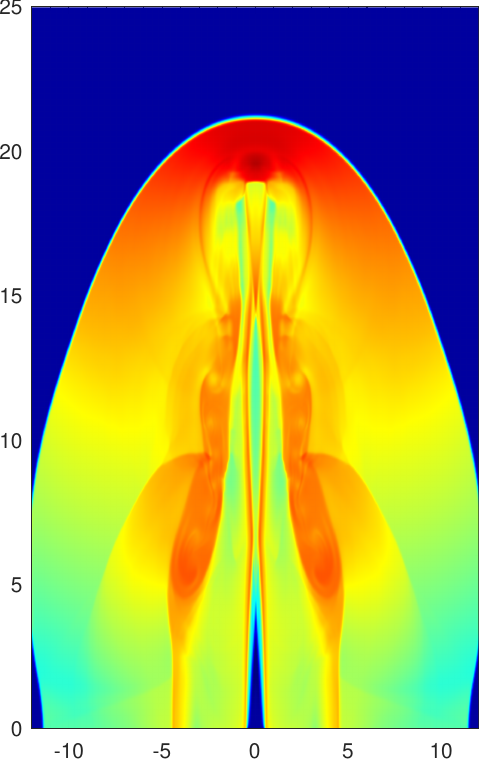}\label{fig:Jet_10}
			}  
			\subfloat[$\log \rho$]{
				\includegraphics[width=0.3\textwidth]{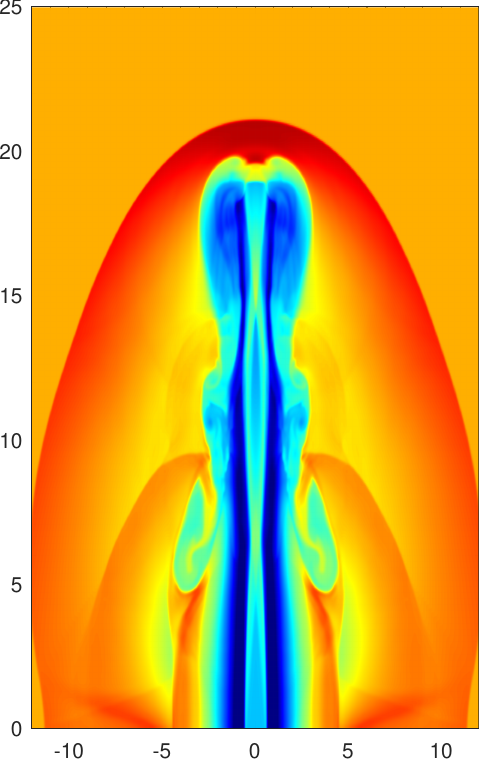}\label{fig:Jet_20}
			}  
			\subfloat[$|{\bf B}|$]{
				\includegraphics[width=0.3\textwidth]{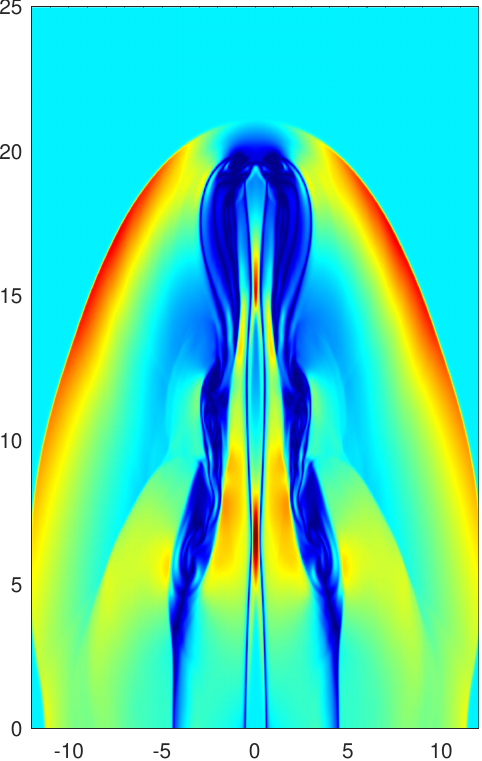}\label{fig:Jet_30}
			}
			\captionsetup{font=small}
			\caption{Astrophysical jet simulated using our integrated PCP NR-DG scheme.}
			\label{fig:jet}
		\end{figure}

In all four numerical examples simulated in Section \ref{sec:DGexamples}, we employed the values from the previous time step as initial guesses for the GR-2DNR and NG-2DNR methods. This strategy is widely used in the literature \cite{gammie2003harm,cerda2008new,noble2006primitive,siegel2018recovery}. However, its effectiveness is unpredictable and challenging to evaluate \cite{kastaun2021robust}, as the initial guess is not chosen deterministically. While this strategy may provide a good initial estimate when the solution is smooth and the time step size is sufficiently small, it can lead to significant discrepancies from the exact  root if the time step is not adequately small or if the solution exhibits discontinuities. 
For the CF-2DNR method, we also used the values from the previous time step as initial guesses; if this approach leads to divergence or failure, we then use the initial guess suggested in \cite{cerda2008new}. 
We observed that the DG schemes combined with either the GR-2DNR or NG-2DNR method failed to progress beyond the first time step in all four tested cases. Additionally, the DG scheme with the CF-2DNR method could not simulate the blast problem beyond $t=0.009$ and quickly failed within the first time step in the jet simulation. 

Considering these results, alongside those from Section \ref{sec:randomtestgamma}, it is clear that the PCP NR method stands out in terms of efficiency and robustness when compared to other primitive variable solvers tested. Furthermore, our numerical experiments indicate that the NH-FP method without the Aitken acceleration  \cite{newman2014primitive} exhibited notable robustness and might be both PCP and convergent. Additionally, the PL-Brent method \cite{palenzuela2015effects} and the MM-1DNR method \cite{mignone2007equation} showed no instances of failure. However, the convergence of the NH-FP and MM-1DNR methods has not yet been proven in theory.

\section{Conclusions}\label{sec:conclusion}

This paper has presented an important advancement in relativistic magnetohydrodynamics (RMHD), tackling a fundamental and long-standing challenge faced by all conservative schemes: the robust and efficient recovery of primitive variables from conservative ones. Our effort has led to the physical-constraint-preserving (PCP) and provably convergent Newton--Raphson (NR) method with a proven quadratic convergence rate. 
The core innovation is a unified approach for the initial guess in the NR method, meticulously designed through sophisticated theoretical analysis. This approach ensures provable convergence and rigorous adherence to physical constraints throughout the NR iterative process, even in ultra-relativistic and strongly magnetized scenarios. 
We have established mathematical theories to analyze the convergence and stability of the PCP NR method. 
The key finding is a crucial inequality pivotal to our analysis of the iterative function's convexity and concavity. 
Our theories have delineated a ``safe" interval for the initial guess, within which the unique positive root of a cubic polynomial always lies. By deriving a real analytical formula for this root and combining it with a more cost-effective initial value, we have obtained an efficient and robust initial guess. 
The broad applicability of the PCP NR method allows its integration into any conservative RMHD numerical scheme. 
It has been successfully incorporated into a PCP discontinuous Galerkin (DG) framework, yielding fully PCP NR-DG schemes that preserve the physicality of both the computed conservative and primitive variables for RMHD. 
The efficiency and robustness of the PCP NR method have been validated through various numerical experiments, including random tests and challenging ultra-relativistic simulations, demonstrating its superiority over other existing solvers. 
Empirical tests reveal that the PCP NR method achieves near-machine accuracy within an average of just five iterations. 
Our findings and insights can inform and enhance the development of other PCP convergent solvers and integrated fully PCP schemes, potentially bringing a broad impact to RMHD.

\appendix

\section{Computational details}\label{sec:detail} 
The pseudocode of the PCP NR method is provided in Algorithm \ref{algor:2}. 
Note that utilizing the formulation \eqref{fu:express2} to implement the function $\mathcal W(\xi)$ is important for large-scale problems with high Lorentz factors or strong magnetic fields. Although the formulations \eqref{eq:DefWxi} and \eqref{fu:express2} for $\mathcal W(\xi)$ are theoretically equivalent, their numerical performance can vary notably due to round-off errors. Figure \ref{fig:fuexpress} demonstrates this by comparing the numerical profiles of ${\mathcal F}(\xi)$ using the two different formulations for $\mathcal W(\xi)$. The profile from \eqref{eq:DefWxi} shows substantial oscillations that could lead to NR iteration instabilities. In contrast, the profile obtained using \eqref{fu:express2} maintains the smoothness and monotonicity of ${\mathcal F}(\xi)$, enhancing stability and reliability.

\renewcommand\baselinestretch{0.86}

\begin{algorithm}[htbp]
	\caption{PCP NR method computing $\xi_*=\rho h W^2$ from a given ${\bf U} \in {\mathcal G}$.}\label{algor:2}
	{\small 
		\begin{algorithmic}

			\State  Compute $m$, $B$, $\tau$, $\beta_1$, $\beta_2$, $\alpha_1$, $\alpha_2$, and $\eta$ by  \eqref{key331}--\eqref{key331b};   Set $\gamma_0=\frac{\gamma-1}{\gamma}$; Compute $\xi_d$ by \eqref{eq:xi_d}.
			
			
			\State $\xi_0\leftarrow\xi_d$; ~ $\mathcal W\leftarrow 1/\sqrt{\frac{(\xi_0 + \alpha_2)(\eta + m)}{\eta^2}  + \beta_1\left(\frac{1}{\eta^2} - \frac{1}{\xi_0^2}\right)}$;\Comment{\fontsize{7.5pt}{0pt}Set $\xi_0=\xi_d$ and calculate $\mathcal W(\xi_d)$}
			
			\State $f_1\leftarrow \xi_0-\gamma_0\left(\frac{\xi_0}{\mathcal W^2}-\frac{D}{\mathcal W}\right)-\frac12\left(\frac{B^2}{\mathcal W^2}+\frac{\tau^2}{\xi_0^2}\right)+\alpha_1$; \Comment{\fontsize{7.5pt}{0pt}Calculate ${\mathcal F}(\xi_d)$}
			
			\State flag$\leftarrow$True; \Comment{\fontsize{7.5pt}{0pt}The flag checks if ${\mathcal F}(\xi_d) \le 0$}
			
			\If{$f_1>0$} \Comment{\fontsize{7.5pt}{0pt}If ${\mathcal F}(\xi_d) >0$, set  $\xi_0=\xi_c$}
			\State $a_0\leftarrow-0.5(B^2D^2+\tau^2)$; $\delta\leftarrow27a_0+4\alpha_1^3$; 
			\If {$\delta>0$}
			\State $\theta\leftarrow\arccos\left(1+\frac{13.5a_0}{\alpha_1^3}\right)$; ~~
			$\xi_0\leftarrow-\frac{\alpha_1}{3}\left(1-2\cos\left(\frac{\theta}{3}-\frac{\pi}{3}\right)\right)$;
			\Else 
			\State $X_1\leftarrow \alpha_1^3+13.5a_0$; ~~ $X_2\leftarrow 1.5\sqrt{3a_0\delta}$; ~~ 
			$\xi_0\leftarrow -\left(\alpha_1+\sqrt[3]{X_1+X_2}+\sqrt[3]{X_1-X_2}\right)/3$;
			\EndIf
			\State flag$\leftarrow$False; 
			\EndIf
			
			\State $N_{\tt osc} \leftarrow 0$; $\xi_1\leftarrow \xi_0$; $f_0\leftarrow0$;
			
			\Do\Comment{\fontsize{7.5pt}{0pt}NR iteration loop}
			\State $\eta\leftarrow \xi_1+B^2$; $\varphi_a\leftarrow -\left(\frac{\beta_1}{\xi_1^3}+\frac{\beta_2}{\eta^3}\right)$; 
			
			\If{flag}		
			\State flag$\leftarrow$False \Comment{\fontsize{7.5pt}{0pt}If $\xi_0=\xi_d$, then $f_1={\mathcal F}(\xi_d)$ has already been computed}
			\Else
			\State $\mathcal W\leftarrow 1/\sqrt{\frac{(\xi_1 + \alpha_2)(\eta + m)}{\eta^2}  + \beta_1\left(\frac{1}{\eta^2} - \frac{1}{\xi_1^2}\right)}$; 
			$f_1\leftarrow \xi_1-\gamma_0\left(\frac{\xi_1}{\mathcal W^2}-\frac{D}{\mathcal W}\right)-\frac12\left(\frac{B^2}{\mathcal W^2}+\frac{\tau^2}{\xi_1^2}\right)+\alpha_1$; 
			\EndIf
			
			\State $df\leftarrow 1+B^2\varphi_a+\frac{\tau^2}{\xi_1^3}-\gamma_0\left(\frac{1}{\mathcal W^2}-2\xi_1\varphi_a+D\mathcal W\varphi_a\right)$; ~~ 
			$\xi_0\leftarrow \xi_1$; ~~ $\xi_1\leftarrow \xi_1-\frac{f_1}{df}$; 
			
			\If{$f_0f_1<0$}
			\State $N_{\tt osc} \leftarrow N_{\tt osc}+1$; \Comment{\fontsize{7.5pt}{0pt}$N_{\tt osc}$ is used to count the NR oscillations (cf.~\cite{flocke2015algorithm})}
			\EndIf
			
			\State $f_0\leftarrow  f_1$;
			\doWhile{$|\xi_0-\xi_1|>\epsilon_{\tt target}$ $\& $ $N_{\tt osc} \leq3$} \Comment{$\epsilon_{\tt target}$ is the given target accuracy}
			
			\Ensure{$\xi_1$ as an approximation to $\xi_*$.}
		\end{algorithmic}
	}
\end{algorithm}

\vspace{-5mm}

\begin{figure}[!htb]
	\centering
		\includegraphics[width=0.39\textwidth]{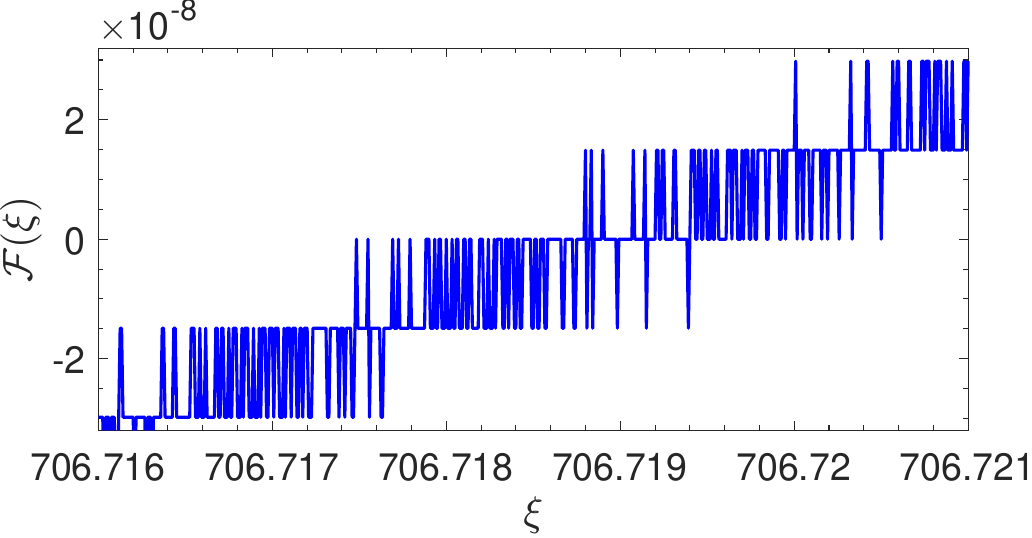}\label{fig:fuexpress1}
~~~~~~~
		\includegraphics[width=0.39\textwidth]{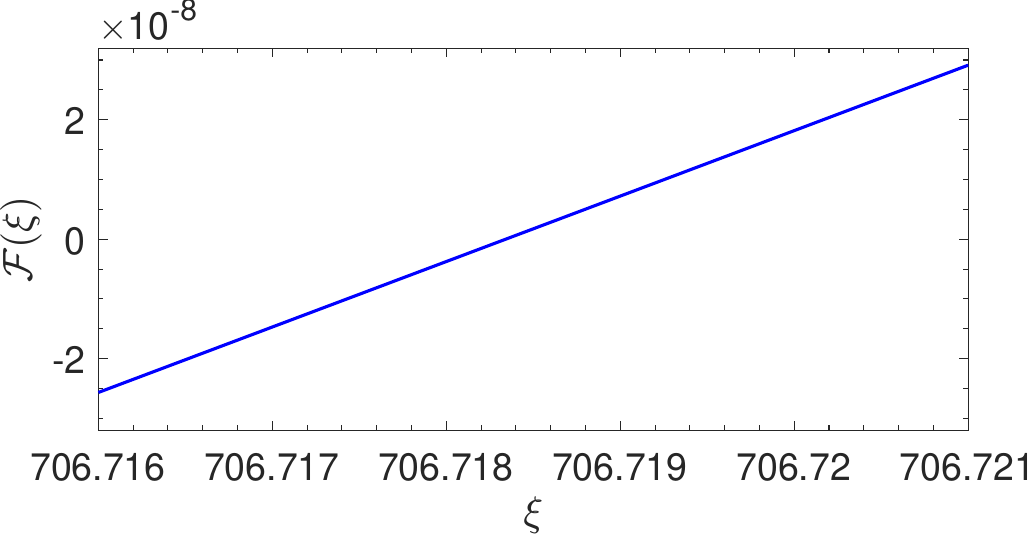}\label{fig:fuexpress2}
	\caption{\small The profiles of 
		${\mathcal F}(\xi)$ with $\mathcal W(\xi)$ computed by \eqref{eq:DefWxi} (left) and \eqref{fu:express2} (right), respectively. Here we set $\gamma=5/3$, $D=1$, $E=10^{8}$, $m^2=9.999999999\times 10^{15}$, $B=10^4$, and $\tau =1$. 
	}\label{fig:fuexpress}
\end{figure}


\begin{changemargin}{-0.2cm}{-0.2cm}  %
	
	\renewcommand\baselinestretch{0.86}
\bibliography{ref}
\bibliographystyle{siamplain}

\end{changemargin}

\newpage

\end{document}